\documentclass[11pt]{amsart}
\usepackage[parfill]{parskip}    
\usepackage{fullpage}
\usepackage{graphicx}
\usepackage{amssymb,amsmath,latexsym}
\usepackage{tikz}
\usetikzlibrary{arrows,decorations.markings,decorations.pathreplacing,calc,matrix,intersections}
\tikzset{->-/.style={decoration={markings,mark=at position #1 with {\arrow{>}}},postaction={decorate}}}

\usepackage{epstopdf}
\usepackage{ifpdf}
\usepackage{color}
\usepackage{float}
\usepackage{amscd,amsmath, mathabx}
\usepackage{amssymb,amsmath,latexsym,color,enumerate,tikz}
\usepackage[all]{xypic}       
\usepackage[varg]{pxfonts}
\usepackage{amsmath,amsthm,amsfonts,amssymb,fancyhdr,graphics,relsize,tikz-cd,mathtools,faktor,tikz,changepage}
\usepackage{graphicx}
\usepackage{placeins}

\usepackage{dsfont}
\definecolor{red}{rgb}{1,0,0}

 \definecolor{darkgreen}{rgb}{0, .7, 0}

 \definecolor{purple}{rgb}{.7, 0, 1}

\newcommand{\F}{{\mathbb{F}}}
\newcommand{\Z}{{\mathbb{Z}}}

\newcommand{\C}{{\mathcal{C}}}
\def\G{{\mathcal G}}
\def\H{{\mathcal H}}

\tikzset{mynode/.style={draw,circle,fill=black,inner sep=2pt,outer sep=0.5pt}}

\newcommand{\bigcupdot}{\hspace{9pt}\cdot \hspace{-9pt}\bigcup}
\newcommand{\cupdot}{\mathbin{\mathaccent\cdot\cup}}

\newtheorem{theorem}{Theorem}[section]
\newtheorem*{theorem*}{Theorem}
\newtheorem*{lemma*}{Lemma}
\newtheorem{proposition}[theorem]{Proposition}
\newtheorem{lemma}[theorem]{Lemma}
\newtheorem{corollary}[theorem]{Corollary}

\theoremstyle{definition}
\newtheorem{definition}[theorem]{Definition}

\theoremstyle{remark}
\newtheorem{remark}[theorem]{Remark}

\newtheorem{para}[theorem]{}
\usepackage{hyperref}

\begin{document}
\title{Finitely generated normal pro-$\C$ subgroups in  right angled Artin pro-$\C$ groups}
\author{Dessislava H. Kochloukova, Pavel A. Zalesskii}

\maketitle

\begin{abstract} Let $\C$ be a class of finite groups closed for subgroups, quotients groups and extensions. Let $\Gamma$ be a  finite  simplicial graph and $G = G_{\Gamma}$ be the corresponding pro-$\C$ RAAG. We show  that if $N$ is a non-trivial finitely generated, normal, full pro-$\C$ subgroup of  $G$ then $G/ N$ is finite-by-abelian. In the pro-$p$ case we show a criterion for $N$ to be of type $FP_n$ when $G/ N \simeq \mathbb{Z}_p$. Furthermore for $G/ N$  infinite  abelian we  show that $N$ is finitely generated if and only if   every  normal closed  subgroup $ N_0 \triangleleft G$ containing $N$  with $G/ N_0 \simeq \mathbb{Z}_p$ is finitely generated. For $G/ N$  infinite abelian with $N$ weakly discretely embedded in $G$ we show that $N$ is of type $FP_n$ if and only if   every $ N_0 \leq G$ containing $N$ with $G/ N_0 \simeq \mathbb{Z}_p$  is of type $FP_n$.
\end{abstract}

\section{Introduction}

In this paper we study right angled Artin pro-$\C$ groups (pro-$\C$ RAAGs), where $\C$ is the class of groups closed for subgroups, quotients and extensions.
Given a   finite simplicial graph $\Gamma$  the discrete RAAG associated to $\Gamma$ is defined by the presentation $$\langle V(\Gamma) \ | \ [v,w] = 1 \hbox{ if } v,w \hbox{ are adjacent in }\Gamma \rangle.$$  If the same presentation is considered in the category of pro-$\C$ groups we have the definition of a pro-$\C$ RAAG associated to $\Gamma$. We will write $G_{\Gamma}$ for the pro-$\C$ RAAG associated with the simplicial graph $\Gamma$. Obviously $G_{\Gamma}$ is the pro-$\C$ completion of the discrete RAAG associated to $\Gamma$.

Though there is a vast literature on abstract RAAGs,   nothing is known about the class of  pro-$\C$ RAAGs and very little about pro-$p$ RAAGs.
In \cite{Lo} Lorensen showed that for a discrete RAAG $G$ associated to a graph $\Gamma$  the canonical map $G \to \widehat{G}_p \simeq G_{\Gamma}$   into its pro-$p$ completion,
induces an isomophism $H^n(\widehat{G}_p, \mathbb{F}_p) \to H^n(G, \mathbb{F}_p)$ for any $n$.  In \cite{K-W} Kropholler and Wilkes proved that  discrete RAAGs are distinguished by their pro-$p$ completions. Snopce and the second author \cite{SZ} gave characterization of coherent pro-$p$ RAAGs as well as  characterized those that appear as the maximal pro-$p$ Galois group $G_K(p)$ of some field $K$ containing a primitive
$p$-th root of unity.

Recently in \cite{CR-Z} Casal-Ruiz and Zearra  studied finitely generated, full, normal subgroups  in discrete RAAGs and showed they are virtually co-abelian. The methods used in \cite{CR-Z} are geometric and involve actions of discrete groups on trees  and the existence of WPD elements associated to those actions, where the trees come from decompositions of the original group as free amalgamated product. Here we are interested in establishing a similar result in the category of pro-$\C$ groups using the theory of pro-$\C$ groups acting on pro-$\C$ trees (see \cite{Z-M 89b}, \cite{ZM90} and \cite{R}). Though in pro-$\C$ context the Bass-Serre theory does not work in its full generality  and there is no pro-$\C$ version of WPD elements,  we prove in Theorem A a pro-$\C$ version of the Casal-Ruiz and Zearra result.

 A pro-$\C$ subgroup $N$ of a pro-$\C$ RAAG $G_{\Gamma} = G_{\Gamma_1} \times \ldots \times G_{\Gamma_k}$, where each $G_{\Gamma_i}$ cannot be further decomposed as $G_{\Gamma_{i,1} } \times G_{\Gamma_{i,2}}$, is called full if $N \cap G_{\Gamma_i} \not= 1$ for every $1 \leq i \leq k$.

\medskip
{\bf Theorem A} {\it Let $\Gamma$ be a  finite  simplicial graph and $G = G_{\Gamma}$ be the corresponding pro-$\C$ RAAG. Suppose that $N$ is a non-trivial, finitely generated, normal, full pro-$\C$ subgroup of  $G$. Then $G/ N$ is  finite-by-abelian, in particular is abelian-by-finite}.

\medskip
In Theorem \ref{graph-product} a more general version of Theorem A is proved for graph products of pro-$\C$ groups. Both results use the following general  result on finitely generated normal subgroups of pro-$\C$ groups acting on pro-$\C$ trees that is of independent interest (see Section 4 for more details).

{\bf Theorem B} {\it Let $G$ be a pro-$\C$ group acting faithfully and irreducibly  on a pro-$\C$ tree $T$ with $|T/G|<\infty$ and let $N$ be a finitely generated, non-trivial normal pro-$\C$ subgroup $N$ of $G$. Then  $T/ N$ is finite.}

\bigskip
The proof of the abstract version of Theorem B  uses essentially the fact that a  finitely generated subgroup $H$ of a group acting on a tree has an $H$-invariant subtree on which $H$ acts cofinitely. This is not true however in the pro-$\C$ case. Thus we needed to find a different argument that contributes essentially to the pro-$\C$ version of the Bass-Serre theory.

In the second part of the paper we restrict our attention to pro-$p$ RAAGs. In section 5
 we give a characterization of finitely generated normal subgroups of a pro-$p$ RAAG.

\medskip
{\bf Theorem C} {\it Let $\Gamma$ be a  finite  simplicial graph and $G = G_{\Gamma}$ be the corresponding pro-$p$ RAAG. Suppose that $N$ is a non-trivial, finitely generated, normal,  pro-$p$ subgroup of  $G$. Then   $N$ is finitely generated if and only if every normal subgroup $ N_0 \triangleleft G$ containing $N$ with $G/ N_0 \simeq \mathbb{Z}_p$  is finitely generated.}

Having this characterization the question of a description of normal subgroups with cyclic quotient arises. Suppose $\chi : G = G_{\Gamma} \to \mathbb{Z}_p$ is an epimorphism of pro- $p$ groups. Set $\Gamma(\chi)$ to be the induced subgraph of $\Gamma$ spanned  by the  vertices $\{ v \in V(\Gamma) \ | \ \chi(v) \not= 0 \}$. We say that $\Gamma(\chi)$ is dominant in $\Gamma$ if every vertex from $V(\Gamma)  \setminus V(\Gamma(\chi))$ is linked by edge with a vertex from $V(\Gamma(\chi))$.

{\bf Theorem D}  {\it $N = Ker (\chi)$  is finitely generated if and only if $\Gamma(\chi)$ is connected and dominant in $\Gamma$.}

In the last two sections  we study when a normal, coabelian pro-$p$ subgroup $N$ of a pro-$p$ RAAG $G$ has homological type $FP_n$. By definition this means that the trivial $\mathbb{Z}_p[[N]]$-module $\mathbb{Z}_p$ has a projective resolution with all modules in dimension $\leq n$ finitely generated. This is equivalent to the pro-$p$ homology $H_i(N
, \mathbb{F}_p)$ being finite for all $i \leq n$, see \cite{K2}.
In \cite{B-G} Bux and Gonzalez showed a homological criterion proved by geometric methods, that classifies when a coabelian normal subgroup of a discrete RAAG is of homological type $FP_n$. A discrete group $G$ has homological type $FP_n$ if the trivial $\mathbb{Z} G$-module $\mathbb{Z}$ has a projective resolution with all projectives finitely generated in dimension at most $n$. For general discrete groups $G$ there is no equivalence between the property $FP_n$ and finite generation of homologies $H_i( G, \mathbb{Z})$ for $i \leq n$. But in the case of a discrete RAAG $G$ and $N$ a normal subgroup of $G$ with $G/ N \simeq \mathbb{Z}$, Papadima and Suciu showed in \cite{P-S} that $N$ is of type $FP_n$ if and only if $H_i(N, k)$ is finite dimensional (over $k$) for every field $k$ and $ 1 \leq i \leq n$. This was latter generalized in \cite{K-MP} by Kochloukova and Martinez-Perez for every  $N$ coabelian in $G$ i.e. $G/N$ is abelian.

In order to state Theorem E we need several definitions.

 We fix a linear order in $ V (\Gamma)$. A clique of $\Gamma$ is a subset $S$ of $V(\Gamma)$ such that every two elements of $S$ are linked by an edge in $\Gamma$. We allow $S$ to be the empty set and denote by $|S|$ the cardinality of $S$. By definition the flag complex $\Delta(\Gamma)$ associated to the graph $\Gamma$ is the complex obtained from $\Gamma$ by gluing
a simplex $(v_1, \ldots , v_n)$ for every non-empty clique $\{ v_1, \ldots, v_n \}$ of $\Gamma$, where $ v_1 < v_2 < \ldots < v_n$. For a clique $S$ of $\Gamma$ the link $lk_{\Delta(\Gamma)} (S)$ is the subcomplex of $\Delta(\Gamma)$ spanned by all simplices $(v_1, \ldots, v_n)$ such that $S \cup \{ v_1, \ldots, v_n \}$ is a clique, $S \cap \{ v_1, \ldots, v_n \} = \emptyset$. We set $lk_{\Delta(\Gamma(\chi))} (S) = lk_{\Delta(\Gamma)} (S) \cap \Delta(\Gamma(\chi))$.

 Let $N$ be a pro-$p$ coabelian subgroup of a pro-$p$ RAAG $G$ and let $G^{dis}$ be the corresponding discrete RAAG, thus $G$ is the pro-$p$ completion of $G^{dis}$.  We call $N$ discretely embedded in $G$ if there is a subgroup $N^{dis}$ of $G^{dis}$ such that $G^{dis}/ N^{dis}$ is abelian and $N$ is the closure of $N^{dis}$ in $G$. We call $N$ weakly discretely embedded in $G$ if there is an automorphism $\varphi$ of $G$ such that $\varphi(N)$ is discretely embedded in $G$.

\medskip
{\bf Theorem E} {\it  Let $\Gamma$ be a finite simplicial graph and $G = G_{\Gamma}$ be the corresponding pro-$p$ RAAG. Let $N$ be a normal pro-$p$ subgroup of $G$ such that $G/ N$ is  infinite abelian. Then the following holds :

(a) suppose that $N = Ker(\chi)$, where $\chi : G \to \mathbb{Z}_p$ is an epimorphism of pro-$p$ groups.
Then $N$ is of type $FP_n$ if and only if $lk_{\Delta(\Gamma(\chi))} (S)$ is $(n-1- |S|)$-acyclic over $\mathbb{F}_p$ for every clique $S$ from $\Gamma \setminus \Gamma(\chi)$  with $|S| \leq n-1$. For $S = \emptyset$ this
translates to the flag complex $\Delta ( \Gamma(\chi))$ is $(n-1)$-acyclic over $\mathbb{F}_p$.

(b)  assume further that  $N$ is weakly discretely embedded in $G$. Then $N$ is of type $FP_n$ if and only if  every   $ N_0 \subseteq G$ containing $N$ with $G/ N_0 \simeq \mathbb{Z}_p$  is $FP_n$.

}

\medskip

 In  part (b) from Theorem E we  assume that   $N$ is weakly discretely embedded in $G$.
We show at the beginning of the proof of Theorem \ref{thm-link} that  if $G/ N \simeq \mathbb{Z}_p$ then $N$ is  weakly discretely embedded in $G$. We prove Theorem E (a) in Theorem \ref{thm-link} and  Theorem E (b) in Theorem \ref{thm-final}.
Both proofs use the fact that the corresponding results hold for discrete RAAGs.
We give a separate proof of weaker version of Theorem E (b)  in Theorem \ref{mild-cond}, that is self-contained in the category of pro-$p$ groups. In general whenever possible, we give proofs within the category of pro-$p$ groups.

\subsection*{Acknowledgement} The first  author was partially supported by Bolsa de produtividade em pesquisa CNPq 305457/2021-7 and Projeto tem\'atico FAPESP 18/23690-6. The second author was partially supported by  CNPq.

\section{Preliminaries}

\subsection{ Pro-$\C$ groups acting on pro-$\C$ trees}
 Let $\C$  be a class of finite groups closed for subgroups, quotients and extensions.
In this section we state some definitions and properties of pro-$\C$ groups acting on pro-$\C$ trees and
free amalgamated pro-$\C$ products. For further
information on this topic we refer the reader to \cite{R-Z2} and \cite{R}.

Following  \cite{R} a  triple $(\Gamma, d_0, d_1)$  is a profinite graph if $\Gamma$  is a
boolean space i.e. totally-disconnected compact space and $$d_0, d_1 : \Gamma \to \Gamma$$ are continuous maps such that $d_i d_j = d_j$ whenever $i,j \in \{ 0,1 \}$. We call
$V (\Gamma) := d_0(\Gamma) \cup d_1(\Gamma)$ the set of vertices of $\Gamma$ and $E(\Gamma) := \Gamma \setminus V (\Gamma)$ the set of edges of $\Gamma$. For $e \in E(\Gamma)$ we call $d_0(e)$ and $d_1(e)$  the initial and
terminal vertices of  the edge $e$. To simplify the notation we write $\Gamma$ for $(\Gamma, d_0, d_1)$. By definition
 $(E^*(\Gamma), *) = (\Gamma/V (\Gamma), *)$ is a pointed profinite quotient space with  a
distinguished point $V (\Gamma)$.

Consider $\mathbb{F}_p[[E^*(\Gamma), *]]$ and $\mathbb{F}_p[[V (\Gamma)]]$ the free profinite $\mathbb{F}_p$-modules
over the pointed profinite space $(E^* (\Gamma), *)$ and over the profinite space $V (\Gamma)$.  For further
information on free profinite modules we refer the reader to \cite{R-Z2}.  Write $\pi(\mathcal{C})$ for the set of primes involved in ${\C}$ and let $p \in \pi(\C)$.
Consider the complex of free profinite $\mathbb{F}_p$-modules
\begin{equation} \label{def1} 0 {\to} \mathbb{F}_p[[E^*(\Gamma), *]] \stackrel{\delta}{\to} \mathbb{F}_p[[V (\Gamma)]] \stackrel{\rho}\to \mathbb{F}_p \to 0
\end{equation}
where $\delta(e) = d_1(e) - d_0(e)$
 for all $e \in E^* (\Gamma)$
and $\rho(v) = 1$ for all $ v \in V (\Gamma)$.
The profinite graph $\Gamma$ is connected if (\ref{def1}) is exact in the middle. The graph $\Gamma$ is a  pro-$\C$ tree if (\ref{def1}) is exact for every $p \in \pi(\C)$. A  pro-$\C$ group $G$ acts on the pro-$\C$
tree $\Gamma$ if it acts continuously on $\Gamma$  such that the action commutes with $d_0$ and $d_1$. We write $G_x$ for the stabilizer of $x \in \Gamma$ in $G$. If the kernel of action is trivial we say that $G$ acts faithfully. If $T$ coincides with its minimal $G$-invariant subtree we say that $G$ acts irreducibly.

Let $A, B$ and $C$ be  pro-$\C$ groups such that $B$ embeds as a  pro-$\C$ subgroup of
$A$ and $C$. Let $ G = A \amalg_B C$ be the amalgamated free  pro-$\C$ product. The amalgamated free  pro-$\C$ product is called proper if $A$ and $C$ embed in $G$. If the
 decomposition of $G$ as an amalgamated free pro-$\C$ product is proper, then there is
a pro-$\C$ tree $T$ on which $G$ acts via left multiplication, where $V (T ) = \{gA \}_{g\in G} \cup \{gC \}_{g \in G}$ and
$E(G) = \{gB \}_{g \in G}$. Here $d_0(gB) = gA$ and $d_1(gB) = gC$. Note
that the quotient $T /G$ is  an edge with two vertices.

In general an amalgamated free pro-$\C$ product $ G = A \amalg_B C$ need not be proper, but if $B$ is procyclic it is proper \cite{R2}. Criterion for a pro-$p$ HNN extension to be proper can be found in \cite{Ch}.

If a pro-$\C$ group acts on a pro-$\C$ tree we use the notation $\widetilde G$ for the subgroup of $G$ generated (topologically) by all vertex stabilizers.

\begin{theorem}\label{quotient tree}  \cite[Prop. 2.5 + Theorem 2.6]{Z-M 89b} or \cite[Prop 4.4.1 + Cor. 4.1.3]{R} \label{prelim1}   Let $G$ be a pro-$\C$ group acting on a pro-$\C$
tree $T$. Then $G/\widetilde G$ acts on the pro-$\C$ tree $T/\widetilde G$ and $G/\widetilde G$ is
a  projective  pro-$\C$ group.
\end{theorem}

\begin{lemma}\label{equal tilde}  \cite[Cor 3.9.3 + Prop 3.10,4]{R} or \cite[Theorems 2.6 + 2.8]{Z-89} \label{prelim2} Let $G$ be a pro-$\C$ group acting on a pro-$\C$ tree $T$ such that $T /G$ is a
pro-$\C$ tree. Then $G = \widetilde G$.
\end{lemma}

\begin{theorem}  \cite[Thm. 2.10]{Z-M 89b} or \cite[Theorem 4.1.8]{R}  \label{subtree} Let $G$ be a pro-$\C$ group acting on a pro-$\C$ tree $T$. Then $$T^G = \{ m \in T \ | \ gm = m \hbox{ for every } g \in G \}$$ is a pro-$\C$ subtree of $T$.
\end{theorem}

\begin{corollary}\label{fixing vertex}  Let $G$ be a pro-$\C$ group acting on a pro-$\C$ tree $T$. Suppose there exists a system of open normal subgroups $U_i$ of $G$  with $\bigcap_i U_i=1$ such that $U_i=\widetilde U_i$. Then $G$ fixes a vertex of $T$.
\end{corollary}

\begin{proof} By Theorem \ref{quotient tree} $T/ U_i$ is a pro-$\C$ tree on which  the finite quotient group $G/U_i$ acts. So by \cite[Theorem 4.1.8]{R} $(T/U_i)^{  G/U_i}\neq \emptyset$. Hence $T^G=\varprojlim_i (T/U_i)^{ G/U_i}\neq \emptyset$ as required.
\end{proof}

\begin{corollary}\label{fixed vertex} Let $G$ be a pro-$\C$ group acting on a pro-$\C$ tree $T$. Suppose either
\begin{enumerate}
\item[(i)] $G$ possesses an open normal subgroup $U$ fixing a vertex;

or

\item[(ii)] Every element of $G$ fixes a vertex.
\end{enumerate}
Then $G$ fixes a vertex of $T$.

\end{corollary}

\begin{proof} To apply Corollary \ref{fixing vertex} in the case (i) take $U_i$ to be all open normal subgroups of $G$ contained in $U$ and in the case (ii) all open normal subgroups of $G$.
\end{proof}

The following lemma will be needed in the next section.

\begin{lemma}\label{intersection}
Let $G$ be a pro-$\C$ group acting on a pro-$\C$ tree $T$. Let $\mathcal{U}$ be a directed set of some open pro-$\C$ subgroups of $G$ and  $N = \cap_{U \in \mathcal{U}} U$. Then $\widetilde{N} = \cap_{U \in \mathcal{U}} \widetilde{U}$.
\end{lemma}

\begin{proof}
Let $M =  \cap_{U \in \mathcal{U}} \widetilde{U}$. Then $T/M$ is the inverse limit of the pro-$\C$ trees $T/ \widetilde{U}$ (see Theorem \ref{quotient tree}), hence $T/ M$ is a pro-$\C$ tree by \cite[Lemma 2.6]{R-Z}. By  Theorem \ref{quotient tree}  $T_0 = T/ \widetilde{N}$ is a pro-$\C$ tree and $J=M/ \widetilde{N}$ acts freely on $T_0$ . But $T_0/ J=T/M$ and so    by Lemma \ref{equal tilde}  we have $J = \widetilde{J}$. Since $J$ acts freely on $T_0$  we have $\widetilde{J} = 1$. Thus $J = 1$.
\end{proof}

\subsection{Finite graphs of pro-$\C$ groups}

In this subsection we recall the definition of a finite graph of pro-$\C$ groups $(\G,\Gamma)$ and its fundamental pro-$\C$ group $\Pi_1(\G, \Gamma)$.
When we say that $(\G,\Gamma)$ is a finite graph of pro-$\C$ groups we mean that it contains the data of the
underlying finite graph, the edge pro-$\C$ groups, the vertex pro-$\C$
groups and the attaching continuous maps. More precisely,

\begin{definition}
Let $\Gamma$ be a connected finite graph. A {\em graph of pro-$\C$ groups} $(\G,\Gamma)$ over
$\Gamma$ consists of  specifying a pro-$\C$ group $\G(m)$ for each $m\in \Gamma$ (i.e., $\G= \bigcupdot_{m\in\Gamma} \G(m)$), and continuous monomorphisms
$\partial_i: \G(e)\longrightarrow \G(d_i(e))$ for each edge
$e\in E(\Gamma)$, $i=1,2$. Moreover,
  \begin{enumerate}
\item A {\em morphism} $(\G,\Gamma) \rightarrow (\H,\Delta)$ of graphs of pro-$\C$ groups
 is a pair
$(\alpha,\bar\alpha)$  of maps, with
 $\alpha:\G\longrightarrow\H$ a continuous map, and  $\bar\alpha:\Gamma\longrightarrow
\Delta$ a morphism of graphs, and such that $\alpha_{\G(m)}:\G(m)\longrightarrow
\H(\bar\alpha(m))$ is a homomorphism for each $m\in \Gamma$ and  which commutes with the appropriate $\partial _i$. Thus the diagram

$$\xymatrix{
\G\ar@{->}^\alpha[rr]\ar@{->}^{\partial_i}[d] & &\H\ar@{->}^{\partial_i}[d]\\
\G\ar@{->}^{\alpha}[rr] & &\H }$$ is commutative.

\item We say that $(\alpha, \bar\alpha)$ is a {\em monomorphism} if both $\alpha,\bar\alpha$ are injective. In this case its image will be called a {\em subgraph of groups} of $(\H, \Delta)$. In other words, a  subgraph of groups of  a graph of pro-$\C$-groups
  $(\G,\Gamma)$ is a graph of groups $(\H,\Delta)$, where $\Delta$ is a
subgraph of $\Gamma$ (i.e., $E(\Delta)\subseteq E(\Gamma)$ and
$V(\Delta)\subseteq V(\Gamma)$, the maps $d_i$ on $\Delta$ are the
restrictions of the maps $d_i$ on $\Gamma$), and for each $m\in\Delta,$
$\H(m)\leq \G(m)$.
\end{enumerate}
\end{definition}

\begin{definition}{\bf Definition of the fundamental pro-$\C$ group.}  \label{def}
In \cite[paragraph (3.3)]{Z-M 89b},  the fundamental pro-$\C$ group
 $G =  \Pi_1(\G,\Gamma)$ is  defined explicitly in terms of generators and relations
 associated to a chosen  maximal subtree $D$. Namely
 \begin{equation} \label{presentation} G=\langle
 \G(v), t_e\mid v\in V(\Gamma), e\in E(\Gamma), t_e=1 \ {\rm for}\  e\in D, \partial_0(g)=t_e\partial_1(g)t_e^{-1},\  {\rm for}\ g\in \G(e)\rangle
\end{equation}
I.e., if one takes the abstract fundamental group $G_0=\pi_1(\G,\Gamma)$
(cf. \cite[\S5.1]{Serre-1980}),
then $$\Pi_1(\G,\Gamma)=\varprojlim_N G_0/N,$$ where $N$ ranges over
all normal subgroups of $G_0$  such that $G/ N \in \C$ and with $N\cap
\G(v)$ open in $\G(v)$ for all $v\in V(\Gamma)$. Note that this last
condition is automatic if $\G(v)$ is finitely generated (as a
pro-$\C$-group) by \cite[Section 4.8]{R-Z2}.
   It is also proved in \cite{Z-M 89b}
that the definition given above is independent on the choice of
the maximal subtree $D$. Then $\Pi_1(\G,\Gamma)$ is the pro-$\C$ completion of $G_0$.  And so if $G_0$ is residually $\C$ and each $\G(v) \in \C$ then $\Pi_1(\G,\Gamma)$ contains open free pro-$\C$ group ( because $G_0$ contains a free abstract normal  subgroup $F$ with $G_0 / F \in \C$).
\end{definition}

The main examples of $\Pi_1(\G,\Gamma)$ are an amalgamated free pro-$\C$
product $G_1\amalg_H G_2$ and an HNN-extension ${\rm HNN}(G,H,t)$ that correspond to the cases of $\Gamma$ having one edge and either two vertices or only one vertex, respectively.

\begin{definition} \label{proper} We call the graph of groups $(\G,\Gamma)$ {\em proper}\footnote{{\em injective} in the terminology of  \cite{R}.} if the natural map
  $\G(v)\to \Pi_1(\G,\Gamma)$ is an embedding for all $v\in V(\Gamma)$.
  \end{definition}

\begin{remark}\label{remark proper} In the pro-$\C$ case a graph of groups $(\G,\Gamma)$ is not always {proper}. However, the
vertex and edge groups can always be replaced by their images  in
$\Pi_1(\G, \Gamma)$ so that $(\G,\Gamma)$ becomes proper and  $\Pi_1(\G,
\Gamma)$ does not change. Thus, throughout the paper, we shall only
consider  proper graphs of pro-$\C$ groups and always identify vertex and edge groups of $(\G,\Gamma)$ with their images in $\Pi_1(\G,\Gamma)$. In particular, all our free amalgamated pro-$\C$ products are proper. \end{remark}

  Note that $(\G,\Gamma)$ is proper if and only if  $\pi_1(\G,\Gamma)$  is  residually  $\C$. In particular, edge and vertex groups will be
subgroups of $\Pi_1(\G,\Gamma)$.

\begin{definition}\label{reduced} A finite graph of  pro-$\C$ groups $(\G,\Gamma)$ is said to be {\em
reduced}, if for every  edge $e$ which is not a loop,
neither $\partial_{1}(e)\colon \G(e)\to \G(d_1(e))$ nor
$\partial_{0}(e):\G(e)\to \G(d_0(e))$ is an isomorphism. \end{definition}
\begin{remark} \label{reduced-2} Any finite graph of  pro-$\C$ groups
can be transformed into a reduced finite graph of pro-$\C$ groups by the
following procedure: if $\{e\}$ is an edge which is not a
loop and for which
one of $\partial_0$, $\partial_1$ is an isomorphism,  we can collapse
$\{e\}$ to a vertex $y$. Let
$\Gamma^\prime$ be the finite graph given by
$V(\Gamma^\prime)=\{y\}\cupdot V(\Gamma)\setminus\{d_0(e),d_1(e)\}$
and $E(\Gamma^\prime)=E(\Gamma)\setminus\{e\}$, and let
$(\G^\prime, \Gamma^\prime)$ denote the finite graph of  groups
based on $\Gamma^\prime$ given by $\G^\prime(y)=\G(d_1(e))$ if
$\partial_{0}(e)$ is an isomorphism, and $\G^\prime(y)=\G(d_0(e))$ if
$\partial_{0}(e)$ is not an isomorphism  and leaving the rest of edge and vertex groups unchanged.

This procedure can be
continued until $\partial_{0}(e), \partial_{1}(e)$ are not surjective for
all edges not defining loops.
 Note that
the  reduction process does not change the
fundamental pro-$\C$ group, i.e., one has a canonical isomorphism
$\Pi_1(\G,\Gamma)\simeq \Pi_1(\G_{red},\Gamma_{red})$.
So, if the pro-$\C$ group $G$ is the fundamental group of a finite
graph of pro-$\C$ groups, we may assume that the finite graph of
pro-$\C$ groups is reduced.
\end{remark}


\begin{para}{\bf Standard (universal) pro-$\C$ tree.}\label{standard}
Associated with the finite graph of pro-$\C$ groups $(\G, \Gamma)$ there is
a corresponding  {\em  standard pro-$\C$ tree}  (or universal covering graph)
  $T=T(G)=\bigcupdot_{m\in \Gamma}
G/\G(m)$ (cf. \cite[Proposition 3.8]{Z-M 89b}).  The vertices of
$T$ are those cosets of the form
$g\G(v)$, with $v\in V(\Gamma)$
and $g\in G$; its edges are the cosets of the form $g\G(e)$, with $e\in
E(\Gamma)$; and the incidence maps of $T$ are given by the formulas:

$$d_0 (g\G(e))= g\G(d_0(e)); \quad  d_1(g\G(e))=gt_e\G(d_1(e)) \ \
(e\in E(\Gamma), t_e=1\hbox{ if }e\in D).  $$

 There is a natural  continuous action of
 $G$ on $T$, and clearly $ G\backslash T= \Gamma$.
Remark also that since $\Gamma$ is finite, $E(T)$ is compact.
\end{para}

\subsection{Homological finiteness properties of pro-$p$ groups}

For a pro-$p$ group $G$ we write $\mathbb{F}_p[[G]]$ for the completed group algebra of $G$ with coeficients in $\mathbb{F}_p$. Thus $$\mathbb{F}_p[[G]] =   \underset{ U} {\varprojlim} \mathbb{F}_p[G/U],$$ where the inverse limit is over all open subgroups $U$ of $G$.

Following \cite{K2} a pro-$p$ group $G$ is of homological type $FP_n$ if there is a projective resolution of $\mathbb{Z}_p$ as a trivial $\mathbb{Z}_p[[G]]$-module, where each projective module in dimension $i \leq n$ is finitely generated. By \cite{K2} this is equivalent with $H_i(G, \mathbb{F}_p)$ is finite for all $i \leq n$. Furthermore,  if $G$
is a finitely generated pro-$p$ group and $N$ its  normal closed subgroup of  $G$ with abelian $G/N$, then by \cite{K2} $G$ is of type $FP_n$ if and only if $H_i(N, \mathbb{F}_p)$ is finitely generated as $\mathbb{F}_p [[G/ N]]$-module for every $i \leq n$.

The property $FP_1$ for a pro-$p$ group is equivalent with finite generation in the category of pro-$p$ groups.
The property $FP_2$ for a pro-$p$ group is equivalent with finite presentation ( in terms of generators and relations) in the category of pro-$p$ groups.

\section{Finitely generated normal subgroups in pro-$\C$ groups acting on trees}

  A vertex $v$ of a graph $\Gamma$ is called a pending vertex if  there is only one edge $e$  incident to $v$.

\begin{lemma} \label{pending fictitious}
  Let $G$ be a pro-$\C$ group acting  on a pro-$\C$ tree $T$ with $|T/G|<\infty$. Then  $G=\Pi_1(\G,T/G)$ is the fundamental group of a graph $(\G,T/G)$ of  pro-$\C$  groups, where vertex and edge groups are stabilizers of the corresponding vertices and edges of a connected transversal of $T/G$ in $T$. Moreover, if the  action is irreducible, then the vertex group of a pending vertex   of $T/G$ can not be equal to the incident edge group. \end{lemma}

  \begin{proof}  The first statement is the subject of \cite[Proposition 4.4 combined with 5.4]{ZM90}.  Now if  $v_0$ is a pending vertex of  $\Gamma = T/G$ and  $e_0$ is its incident edge with  $G(e_0)=G(v_0)$ then we can consider a subgraph $\Delta=\Gamma\setminus \{e_0,v_0\} $ of $\Gamma$ and the graph of groups $(\G,\Delta)$ restricted to $\Delta$. Then $G=\Pi_1( \G,\Delta)$.

  Let $\Sigma$ be a connected transversal of $\Delta$ in $T$. Then $G\Sigma$ is a pro-$\C$ subtree of $\Gamma$; indeed using presentation (\ref{presentation}) we see that for the group $G^{abs}$ generated abstractly by $\{\G_v, t_e\mid e\in E(T/G), v\in V(T/G)\}$ one has that $G^{abs}\Sigma$ is connected,   since the corresponding result for discrete groups holds. Indeed $T^{abs} = G^{abs} (\Sigma \cup \{ e_0, v_0 \})$ is an abstract tree, hence connected and since  $v_0$ is a pending vertex and $G(v_0) = G(e_0)$, $G^{abs} \Sigma$ continues being connected.

   Hence   $G\Sigma$ is the  closure of $G^{abs}\Sigma$   in $T$  and  $G\Sigma$ is connected as well. This contradicts irreducibility of the action,  since $G\Sigma$ is a proper pro-$\C$ subtree of $T$.

  \end{proof}

We say that the fundamental group of a graph of pro-$\C$ groups is of dihedral type if $\Gamma$ has only one edge $e$ and  either  $G_e = G_v$ if $e$ is a loop or  $e$ has  vertices $v\neq w$  and $[G_{v} : G_e] = 2 = [G_w : G_e]$. If in a graph of pro-$\C$ groups ( over finite graph) all vertex groups are  in $\C$ we have that its fundamental pro-$\C$ group $G$ is virtually procyclic precisely when  $G$ is of dihedral type.  Indeed using the abstract group $G_0$ from definition \ref{def} we have that $G$ is the pro-$\C$ completion of $G_0$, hence $G$ is virtually procyclic if and only if $G_0$ is virtually cyclic. And $G_0$ is virtually cyclic precisely when the original graph of pro-$\C$ groups is of dihedral type.

\begin{lemma}\label{bound} Let $G=\Pi_1(\G,\Gamma)$ be the fundamental  pro-$\C$ group of a finite reduced graph of finite   $\C$-groups  and let $F$ be   an open  free pro-$\C$ subgroup  of $G$.  Let $T$ be the standard pro-$\C$ tree associated with $G=\Pi_1(\G,\Gamma)$.  Then for any edge $e\in E(T)$ and its incident vertex $v$ one has
\begin{enumerate}
\item[(a)]  $ [G_{v} :  G_e] < 3rank(F)  + 2$;
\item[(b)] If $(\G,\Gamma)$ is not of a dihedral type, then     $ [G : FG_e] < 6  rank(F)$.
\end{enumerate}
\end{lemma}

\begin{proof}  All pro-$\C$ groups  considered in this proof are pro-$\C$ completions $\widehat{A}_{\C}$  of some abstract groups $A$ with the following properties: $A$ has a normal subgroup $M$ of finite cohomological dimension, type $FP_{\infty}$ and $A/ M \in {\C}$. By  \cite[Ch. 9]{Brown}  there is a well-defined ``fractional'' Euler characteristic $\chi(A)$ and we set $\chi(\widehat{A}_{\C}) = \chi(A)$.
The values of such Euler characteristic are rational numbers and are not necessary integers. For an open   subgroup $N$ of $G$ we have that $\chi(N) = [G : N] \chi(G)$.
 Then
$$\chi(G) = \frac{\chi(F)}{[G:F]} = \frac{ 1 - rank(F)}{[G:F]}.$$
Note that $F \cap G_v = 1$ and since every vertex group $G_v$ is finite, $\chi(G_v) = 1/ |G_v|$ and $\chi(G_e) = 1 / | G_e|$ for $v \in V(T), e \in E(T)$.
Then ( for the pro-$p$ case see \cite[Remark 3.6]{HZZ}, the pro-$\C$ case is similar)
$$\frac{ rank(F)-1}{[G:F]}=-\chi(G) = -\sum_{v\in V(\Gamma)}\chi(G_v)+\sum_{e\in E(\Gamma)} \chi(G_e) = \sum_{e\in E(\Gamma)}\frac{1}{| G_{e}|} - \sum_{v\in V(\Gamma)}\frac{1}{| G_{v}|}.$$
Choose $e_0\in E(\Gamma)$ and its incident vertex $v_0\in V(\Gamma)$. To tame the negative term, choose a maximal rooted subtree $D$ of $\Gamma$ with  $v_0$ in $\Gamma$ being the root and direct it away from the root.  Then each vertex $v$ of $\Gamma$ other than the root is the terminal point $t(e)$ of exactly one incoming edge in $D$.

 Consider first the case  $e_0$  is not a loop and   we can choose $D$ to contain  $e_0$.
Thus we can rewrite  the above sum as
\begin{equation}\label{sum}
\frac{ rank(F)-1}{[G:F]}=\frac{1}{| G_{e_0}|}-\frac{1}{| G_{v_0}|}  -\frac{1}{| G_{v_1}|} + (\sum_{e\in E(D)\setminus \{e_0\}}\frac{1}{| G_{e}|}-\frac{1}{| G_{t(e)}|}) +(\sum_{e\in E(\Gamma)\setminus E(D)}\frac{1}{| G_{e}|})
,\end{equation}
 where $v_1$ and $v_0$ are the vertices of $e_0$.
As  summations of the right handside are  non-negative,  $$[G:F](\frac{1}{| G_{e_0}|}-\frac{1}{| G_{v_0}|}-\frac{1}{| G_{v_1}|}) \leq rank(F)-1.$$

Note that for  every $v \in V(\Gamma)$ we have $F \cap G_{v} = 1$, so $G_{v}$ embeds in $G / F$ and $m_{j} |G_{e_0}| = |G_{v_j}|$ divides  $|G/ F|$ for $j = 0$ and $j = 1$. Put $m=MMC(m_0, m_1)$, $a=MCD(m_0,  m_1)$.  Hence $m |G_{e_0}|$ divides $[G:F]$ and we can conclude that
$$  rank(F)-1 \geq  [G:F](\frac{1}{| G_{e_0}|}-
\frac{1}{| G_{v_0}|}-\frac{1}{| G_{v_1}|})  =
 [G:F] (\frac{1}{| G_{e_0}|}-\frac{1}{m_0| G_{e_0}|}-\frac{1}{m_1| G_{e_0}|})  =$$
$$  \frac{[G:F]} {m | G_{e_0}|} (m -\frac{m}{m_0}-\frac{ m}{m_1}) \geq  m -\frac{m}{m_0}-\frac{ m}{m_1}.
$$

Set $m_{j} =  a \widetilde{m}_{j}$. Then
$$
(\widetilde{m}_{0} - 1)  (\widetilde{m}_{1} - 1) -1 = \widetilde{m}_{0}  \widetilde{m}_{1} -  \widetilde{m}_{1} -  \widetilde{m}_{0} \leq a \widetilde{m}_{0}  \widetilde{m}_{1} -  \widetilde{m}_{1} -  \widetilde{m}_{0}    = m  - \frac{m}{m_{0}} - \frac{m }{m_{1}}   \leq rank(F)-1.$$
Then if $\widetilde{m}_{1} > 1$ we conclude that $\widetilde{m}_{0} - 1 \leq
(\widetilde{m}_{0} - 1)  (\widetilde{m}_{1} - 1) \leq rank(F)$. If $\widetilde{m}_{1} = 1$ then $1 \not= m_{1} = a \widetilde{m}_{1} = a $, so
$\widetilde{m}_{0} -1 \leq (a - 1)\widetilde{m}_{0}-1  \leq  a \widetilde{m}_{0}  \widetilde{m}_{1} -  \widetilde{m}_{1} -  \widetilde{m}_{0} \leq rank(F)-1$. Thus in both cases we have
$\widetilde{m}_{0} \leq rank(F)  + 1$, and similarly  $\widetilde{m}_{1} \leq rank(F) +1.$

Finally to prove (a) for edges in $D$ note that
$$
m_{j} = a \widetilde{m}_{j} \leq a \widetilde{m}_{0} \widetilde{m}_{1} \leq rank(F)-1 +
\widetilde{m}_{0} + \widetilde{m}_{1} \leq rank(F)-1 + rank(F)   + 1 + rank(F)   + 1= 3rank(F)  +  1.$$
 Thus we are left to prove (a) for edges $e_0$ that are loops. Then
\begin{equation}\label{sum1}
\frac{ rank(F)-1}{[G:F]}=\frac{1}{| G_{e_0}|}-\frac{1}{| G_{v_0}|}   + (\sum_{e\in E(D)}\frac{1}{| G_{e}|}-\frac{1}{| G_{t(e)}|}) +(\sum_{e\in E(\Gamma)\setminus (E(D) \cup \{ e_0 \})}\frac{1}{| G_{e}|})
,\end{equation}
 together with the fact that the summations are non-negative implies
$$\frac{1}{| G_{e_0}|}-\frac{1}{| G_{v_0}|}\leq \frac{  rank(F)-1}{[G:F]}$$ and
hence $$[G_{v_0}:G_{e_0}] =   \frac{|G_{v_0}|}{|G_{e_0}|}  \leq \frac{  (rank(F)-1)|G_{v_0}|}{[G:F]}+1 \leq rank(F).$$
This finishes case (a).

\medskip
 Now to show (b)  we rewrite the equation (\ref{sum}) as
\begin{equation} \label{new1} rank(F)-1=\frac{[G:F]}{|G_{e_0}|}(1-\frac{| G_{e_0}|}{| G_{v_0}|}  -\frac{| G_{e_0}|}{| G_{v_1}|}) + (\sum_{e\in E(D)\setminus \{e_0\}}\frac{[G:F]}{|G_{e}|}(1-\frac{| G_{e}|}{| G_{ t(e)}|}) +(\sum_{e\in E(\Gamma)\setminus E(D)}\frac{[G:F]}{| G_{e}|})  \end{equation}
and we rewrite (\ref{sum1}) as
\begin{equation} \label{new2} rank(F)-1=\frac{[G:F]}{|G_{e_0}|}(1-\frac{| G_{e_0}|}{| G_{v_0}|}) + (\sum_{e\in E(D)}\frac{[G:F]}{|G_{e}|}(1-\frac{| G_{e}|}{| G_{ t(e)} |}) +(\sum_{e\in E(\Gamma)\setminus (D \cup \{ e_0 \})}\frac{[G:F]}{| G_{e}|})  \end{equation}

Since $(\G,\Gamma)$ is reduced  all terms of summations are positive and so $rank(F)=1$ only if $\Gamma$ is one edge $e_0$ and either $e_0$ is not a loop and
 $[G_{v_0} : G_{e_0}]=2= [G_{v_1}:G_{e_0}]$ or $e_0$ is a loop and $G_{v_0}= G_{e_0}$ (i.e. $G$ is a dihedral type, exluded by the hypothesis).

We consider 3 cases.

1) Suppose that  $e$ is not a loop. Then if $|E(D)| \geq 2$ we can consider (\ref{new1})  for $e \not= e_0$ and assume that $e \in E(D)$.  Then
$$rank(F) - 1 \geq \frac{[G:F]}{|G_{e}|}(1-\frac{| G_{e}|}{| G_{v}|}).$$
 Since $(\G,\Gamma)$ is reduced $G_e \not= G_v$. Then since $|G_e|$ divides $|G_v |$ we conclude that $\frac{|G_v|} {| G_e|} \geq 2$,  hence $(1-\frac{| G_{e}|}{| G_{v}|})^{-1} \leq
 2$, so using  $F \cap G_e = 1$
 $$[G : F G_e] =   \frac{[G:F]}{|G_{e}|}   \leq (rank(F) - 1)  (1-\frac{| G_{e}|}{| G_{v}|})^{-1} \leq 2{(rank(F) - 1)}.$$

 2) Suppose that  $e=e_0$   is not a loop and $E( D) = \{ e_0 \}$.
Then  by (\ref{new1}) we have
$$rank(F) -1 \geq {[G : F G_{e_0}]} (1-\frac{| G_{e_0}|}{| G_{v_0}|}  -\frac{| G_{e_0}|}{| G_{v_1}|}).$$
Since $(\G,\Gamma)$ is reduced and not of dihedral type $ 1-\frac{| G_{e_0}|}{| G_{v_0}|}  -\frac{| G_{e_0}|}{| G_{v_1}|} > 0$.  Then
using that $\frac{| G_{v_0}|} {| G_{e_0}| } $, $\frac{| G_{v_1}|} {| G_{e_0}| } $ are integers, both at least 2 but not simultaneously 2 we conclude that $\frac{| G_{e_0}|}{| G_{v_0}|}  +\frac{| G_{e_0}|}{| G_{v_1}|} \leq \frac{1}{2} + \frac{1} {3} = \frac{5}{6}$, so
$$ [G : F G_e] \leq (rank(F) - 1) (1-\frac{| G_{e_0}|}{| G_{v_0}|}  -\frac{| G_{e_0}|}{| G_{v_1}|})^{-1} \leq 6(rank(F) - 1) $$

3) We are left with the case when  $e$ is a loop. Then we can consider (\ref{new2}) and assume that $e = e_0$. So
$$rank(F) -1 \geq {[G : F G_{e_0}]} (1-\frac{| G_{e_0}|}{| G_{v_0}|}) .$$

If $G_{e_0}\neq G_{v_0}$ then as in Case 1
$$[G : F G_e] \leq (rank(F) - 1)  (1-\frac{| G_{e}|}{| G_{v}|})^{-1} \leq 2{(rank(F) - 1)}.$$

Suppose $G_{e_0}= G_{v_0}$. Suppose that $\Gamma$ posseses an edge $ e_1$  with incident vertex $v_0$ such that if $ e_1$ is a loop then $G_{ e_1} \not= G_{v_0}$. Then from the previous cases one has $$ [G : F G_{e_1}] \leq 6(rank(F) - 1) $$ and so since $G_{ e_1} \leq G_{v_0} = G_{e_0}$ we have $$ [G : F G_{e_0}] = [G : F G_{v_0}] \leq [G : F G_{ e_1}] \leq 6(rank(F) - 1). $$

Thus since $\Gamma$ is connected, we are left with the case where $\Gamma$ is a bouquet and $G_e=G_{v_0}$ for all edges $e$. Since $G$ is not dihedral, we have that $|E(\Gamma)| \geq 2$. Then $G=F_0\rtimes G_{v_0}$ for some not procyclic free pro-$\C$ group $F_0$ and the result follows from the Schrier formula.  Indeed, we can assume $F \leq F_0$, and so
$$[G: F G_{e_0}] = [G : F G_{v_0}] = [F_0 G_{v_0} : F G_{v_0}] = [F_0 : F] = \frac{d(F) - 1}{d(F_0) - 1} \leq d(F) - 1.$$
 \end{proof}

 \begin{proposition}\label{bound1} Let $G$ be a pro-$\C$ group acting irreducibly on  a pro-$\C$ tree with finite vertex stabilizers such that $T/G$ is finite. Suppose $G$ is not virtually  procyclic and let $F$ be   an open free pro-$\C$ subgroup  of $G$. Then for any edge $e\in E(T)$  one has
    $ [G : FG_e] < 6  rank(F)$.
\end{proposition}

\begin{proof} By Lemma \ref{pending fictitious} $G=\Pi_1(\G,\Gamma)$ is the fundamental group of a finite graph of groups  from $\C$ with $G(v)$ not equal to its incident edge group for any pending vertex $v$ of $\Gamma$. So it suffices to prove the  inequality for the edge groups $G(e)$.

We say that an edge $e$  of $\Gamma$ is fictitious if it is not a loop and  $G(e)=G(v)$ for some of its incident vertex $v$. We use induction on the number of fictitious edges in $\Gamma$. If this number is 0, then the result is the subject of Lemma \ref{bound}. Let $e$ be a fictitious edge with vertices $v,w$, i.e. $G(e)=G(v)$ or $G(w)$, say $G(e)=G(v)$. Collapse this edge i.e. substitute the edge with one vertex, whose new vertex group is $G(w)$. Since $v$ is not pending, there exist an edge $e' \not= e$ incident to $v$. By induction hypothesis $ [G : FG(e')] < 6  rank(F)$. Then $ [G : FG(v)] < 6  rank(F)$ and so $ [G : FG(e)]  =   [G : FG(v)] < 6  rank(F)$ as needed.
\end{proof}

One of the main obstacles in the pro-$\C$ vesrion of the Bass-Serre theory is that a finitely generated subgroup acting on a pro-$\C$ tree does not admit an invariant subtree with cofinite action. The next theorem shows that we have cofiniteness of the action if we assume the subgroup to be normal.

\begin{theorem} \label{main} Let $G$ be a pro-$\C$ group acting faithfully and irreducibly  on a pro-$\C$ tree $T$ with $|T/G|<\infty$.
Suppose $G$ has a finitely generated, non-trivial normal pro-$\C$ subgroup $N$.  Then  $T/ N$ is finite  or $G$ is virtually procyclic.
\end{theorem}

\begin{proof}

Suppose that $V$ is an open subgroup of $G$. Then $G/ \widetilde{V}$ acts on the pro-$\C$ tree $T_V = T/ \widetilde{V}$ and since $V/ \widetilde{V}$ acts freely and  is open in $G/ \widetilde{V}$ we conclude that the vertex stabilizers of the action of $G/ \widetilde{V}$ on $T_V$ are finite $\C$-groups. By Lemma \ref{pending fictitious} $G / \widetilde{V}$ is the fundamental group of a finite graph of groups in $\C$, so is the pro-$\C$ completion of the abstract fundamental group $G_{0,V}$ of the same graph of groups. Note that $G_{0, V}$ is a finitely generated, virtually free abstract group and so $V/ \widetilde{V}$ is the pro-$\C$ completion of a free abstract group, hence is a free pro-$\C$ group.

Assume that $G$ is not virtually procyclic.  Then by \cite[Theorem 3.1]{Z-90} or \cite[Theorem 4.2.10]{R} $G$ contains a non-abelian free pro-$p$ subgroup  $F$ acting freely on $T$.   Let $$\mathcal{V} = \{ V \leq_oG\ | \ F \subseteq V \}.$$ Then
$$\bigcap_{V \in \mathcal{V}} \widetilde{V} = \widetilde{F} = 1$$ (cf. Lemma \ref{intersection}). 

We claim that for some $V \in \mathcal{V}$ we have that $V/ \widetilde{V}$ is a non-abelian free pro-$\C$ group.  Indeed assume $V/ \widetilde{V}$ is  procyclic for all $V \in \mathcal{V}$ (including the possibility $V/ \widetilde{V}=1$); then $\underset{ V \in \mathcal{V}} {\varprojlim} V/ \widetilde{V}$  is procyclic (including the possibility to be trivial). Hence the exactness of $ 1 \to \widetilde{V} \to V \to V/ \widetilde{V}$ yields an exact sequence $ 1 \to  \underset{ V \in \mathcal{V}} {\varprojlim} \widetilde{V} \to  \underset{ V \in \mathcal{V}} {\varprojlim} V \to  \underset{ V \in \mathcal{V}} {\varprojlim} V/ \widetilde{V}$, that can be rewritten as $ 1 \to  \cap_{ V \in \mathcal{V}} \widetilde{V} \to
 F = \cap_{ V \in \mathcal{V}} V \to \underset{ V \in \mathcal{V}} {\varprojlim} V/ \widetilde{V}$. Therefore $F$ is procyclic, a contradiction.

Then for the core $U_0  = \cap_{g \in G} V^g $ of $V$ we also have  that $G/\widetilde U_0$ is not virtually procyclic.  By \cite[Proposition 4.2.3 (a)]{R} or \cite[Proposition 2.1]{Z-90} $U_0$ acts irreducibly on $T$ and since if $T/ (N \cap U_0)$ is finite  $T/ N$ is finite, replacing $G$ with $U_0$ and $N$ with $N\cap U_0$ we may assume that $G/\widetilde G$ is free non-abelian  pro-$\C$.

Let $\mathcal{U}$ be the set of all open subgroups $U$ of $G$ that contain $N$. Hence  $N = \bigcap_{U \in \mathcal{U}} U$.
Fix $U \in \mathcal{U}$ and observe that $U/\widetilde U$ is free non-abelian  pro-$\C$.

1) Suppose $N\not\leq\widetilde{U}$ for some $U\in \mathcal U$.
Note that $N_U = N \widetilde{U}/ \widetilde{U}$ is a  finitely generated, non-trivial normal subgroup of $G/ \widetilde{U}$ that is contained in the free non-abelian pro-$\C$ group $U/ \widetilde{U}$, so by \cite[Theorem 8.6.5]{R-Z2} $N_U$  is open in $G/ \widetilde{U}$ and hence is free pro-$\C$; since $d = d(N)  < \infty$ we get that $N_U$ is of rank at most $d = d(N)$.
 Then for any open normal subgroup $V$ of $G$ such that $N\leq V\leq U$, the normal subgroup $N_V=N\widetilde V/\widetilde V$ has finite index in $G/ \widetilde{V}$.
Denote by $G_{e,V}$ the image of $G_e$ in $G_V = G/ \widetilde{V}$ and observe that it is the edge stabilizer of the image of $e$ in $T/\widetilde V$. We can apply Proposition \ref{bound1}
to $G_V $ acting on $T/ \widetilde{V}$
to conclude that for every edge $e$ one has
$$[G : N G_e] = [G/ \widetilde{N} :  G_e\widetilde N/ \widetilde{N}] = sup_{V \leq U} [G_V:  G_{e,V}N_V] < 6 d(N).$$

Note that $|T / N|<\infty$  is equivalent to $[G : N G_e]< \infty$ for every edge $e$.
Thus $T/N$ is finite in this case.

2) Suppose now  $N \subseteq \bigcap_{U \in \mathcal{U}}\widetilde{U}  = \widetilde{N} \subseteq \bigcap_{U \in \mathcal{U}}{U} = N$ (see Lemma \ref{intersection}), so $N = \widetilde{N}$.

If $N$ stabilizes a vertex then by \cite[Theorem 2.12]{Z-M 89b} or \cite[Proposition 4.2.2 (a)]{R} $N$ is in the kernel of the action, contradicting the faithfulness assumption. Thus this can not occur and by Corollary \ref{fixed vertex}
 there is $x \in N$ that  does not fix any vertex. Then $C = \overline{\langle x \rangle}$ acts freely on $T$. Let $$\mathcal{V} = \{ V \ | \ V \hbox{ an open } \hbox{ subgroup of } N, C \subseteq V \}.$$ Then
if $V = \widetilde{V}$ for every $V \in \mathcal{V}$, we have
$$C = \bigcap_{V \in \mathcal{V}} V = \bigcap_{V \in \mathcal{V}} \widetilde{V} = \widetilde{C} = 1$$ (cf. Lemma \ref{intersection}),
a contradiction.
Hence there is $V_0 \in \mathcal{V}$ such that $V_0\not= \widetilde{V_0}$.

Note that for $g \in G$ we have that $[N : V_0^g] = [N : V_0]$ and since $N$ is finitely generated there are only finitely many open subgroups of a fixed index in $N$, so  $V_1 = \bigcap_{g \in G} V_0^g$ is open in $N$. Thus $V_1$ is a finitely generated normal subgroup of $G$.  Note that if $V_1 = \widetilde{V}_1$ then $V_1 = \widetilde{V}_1 \subseteq \widetilde{V}_0 \subseteq V_0$ and $V_0 / \widetilde{V}_0$ is a non-trivial free group, that is a quotient of the finite group $V_0/ V_1$, a contradiction. Hence $V_1 \not= \widetilde{V}_1$ and by  Case 1 the quotient graph
$T/ V_1$ is finite.  Then $T/ N$ is finite.
\end{proof}

The structure theorem for normal subgroups of free constructions of profinite groups was proved in
  \cite[Theorem B]{Z}. The next corollaries strengthen this result.

\begin{corollary}\label{normal} Let $G=\Pi_1(\G,\Gamma)$ be the fundamental group of a finite graph of  pro-$\C$ groups and $N$   a non-trivial finitely generated normal closed subgroup $N$ of $G$. Then $N$ is the fundamental group of a finite graph of pro-$\C$ groups, whose vertex  and edge groups are conjugate to subgroups of vertex and edge groups of $\G$.
\end{corollary}

\begin{proof} If $G$ is virtually procyclic, it is of dihedral type and the statement is obvious. Assume $G$ is not virtually procyclic.  Let $T(G)$ be the standard pro-$\C$ tree on which $G$ acts. If $N$ is in the kernel of the action then there is nothing to prove, so we assume that $N$ is not in the kernel of the action. The action is irreducible (\cite[Proposition 4.2.3 (a)]{R} or \cite[Proposition 2.1]{Z-90}) and killing the kernel of the action we may assume it is faithful.  Then   by Theorem \ref{main} $T(G)/N$ is finite and so by \cite[Proposition 4.4]{ZM90}   $N$ is the fundamental group of a finite graph of pro-$\C$ groups over $T(G)/N$ in the desired form.
\end{proof}

\begin{corollary}\label{free product} Let $G=\coprod_{x\in X} G_x$  be a free pro-$\C$ product and $N$ a non-trivial finitely generated normal  closed subgroup of $G$. Suppose $G$ is not virtually procyclic. Then $N$ is open in $G$ of index $[G:N] < 6 d(N)$. \end{corollary}

\begin{proof}  Decompose $G$ as an inverse limit of finite free pro-$\C$ products $G=\varprojlim_i G_i$, $G_i=\coprod_{x_i\in X_i} G_{x_i}$,  where  $X_i$ are finite and each $G_{x_i}$ is a   finite  $\C$-group. Note that $G_i$ acts on its standard pro-$\C$ tree $T_i$ irreducibly   and its edge stabilizers are trivial.  Let $N_i$ be the image of $N$ in $G_i$. Then $G_i / \widetilde{N}_i$ acts irreducibly on $T_i/ \widetilde{N}_i$, actually $G_i/ \widetilde{N}_i$ is $\coprod_{x_i \in X_i} G_{x_i}/ (G_{x_i} \cap N_i)$ and $T_i/ \widetilde{N}_i$ is the corresponding standard pro-$\C$ tree.  Hence the index of $N_i$ in $G_i$ is less then $6d(N)$ by Proposition \ref{bound1}. Therefore $[G:N] < 6 d(N)$.
\end{proof}

\begin{corollary} Let $G=\Pi_1(\G,\Gamma)$ be the fundamental group of a finite graph of  pro-$\C$ groups and $H$ a  closed subgroup of $G$ possesing    a non-trivial finitely generated  closed normal subgroup $N$ of $G$ not contained in a vertex group $\G(v)$ of $(\G,\Gamma)$.  Then $H$ is the fundamental group of a finite graph of pro-$\C$ groups, whose vertex and edge  groups are conjugate to  subgroups of vertex and edge groups of $\G$.
\end{corollary}

\begin{proof}  As in the proof of Corollary \ref{normal} we may assume that $G$ is not virtually procyclic. Let $T(G)$ be the standard pro-$\C$ tree on which $G$ acts.  The action is irreducible and killing the kernel of the action we may assume it is faithful.  Then   by Theorem \ref{main} $T(G)/N$ is finite and therefore so is $T(G)/H$. Then  by \cite[Proposition 4.4]{ZM90}   $H$ is the fundamental group of a finite graph of pro-$\C$ groups over $T(G)/H$ in the desired form.

\end{proof}

\begin{corollary} Let $G=G_1\amalg_A G_2$ be a  free pro-$p$ product of coherent pro-$p$ groups amalgamating  analytic pro-$p$ subgroup $A$. Let $N$ be a finitely generated normal closed subgroup of $G$. Then $N$ is finitely presented.
\end{corollary}

\begin{proof}  By Corollary \ref{normal}  $N$ is the fundamental group of a finite graph of pro-$p$ groups and since $N\cap A^g$ is finitely generated for any $g\in G$ and  $N$ is finitely generated,  by Mayer-Vietoris argument in homology   $N\cap G_i^g$ is finitely generated and hence is finitely presented  for any $g \in G$. Thus we get a finite presentation for $N$ (see (\ref{presentation})).
\end{proof}

\begin{remark} We do not know whether a condition of normality for $N$ can be dropped.

\end{remark}

\section{Normal subgroups of pro-$\C$ RAAGS}

 For a pro-$\C$ group $G$ and a closed subgroup $K$ we denote by $Core_G(K) = \cap_{g \in G} K^g$ the core of $K$ in $G$.

\begin{lemma} \label{core} Let $G = G_1 \coprod_B G_2$ be a pro-$\C$ product with $B \not= G_2$. Then $Core_G(G_1) = Core_G(B)$.
\end{lemma}

\begin{proof} Note that $G$ acts on the standard pro-$\C$ tree $T$ associated to the above amalgamated product. Let $G_1$ be the stabilizer of the vertex $v_1 \in V(T)$ and $G_2$ be the stabilizer of $v_2 \in V(T)$. Let $g \in G_2 \setminus B$. Then
$M = Core_G(K) \subseteq G_1 \cap G_1^{g^{-1}}$, so $M$ is in the stabilizer of $v_1$ and $g v_2$. Then by Theorem \ref{subtree}  $M$ fixes the maximal pro-$\C$ tree of $T$ that contains $v_1$ and $g v_2$, in particular $M$ fixes an edge i.e. there is $g_0 \in G$ such that $M \subseteq B^{g_0}$
and since $M$ is normal in $G$ we conclude that $M \subseteq B$. Hence $Core_G(G_1) = Core_G(B)$.
\end{proof}

Let $\Gamma$ be a finite (simplicial) graph and for every vertex $v \in V(\Gamma)$ there is a pro-$\C$ group $G(v)$. Consider the corresponding graph product  of pro-$\C$ groups i.e. the pro-$\C$ group given by the presentation
$$G =\mathcal{G} (\Gamma) = \langle \{ G(v) \}_{v \in V(\Gamma)} \ | \ [G(v), G(w)] = 1 \hbox{ if }v, w  \hbox{ are adjacent in } \Gamma \}.$$

We fix a vertex $v \in V(\Gamma)$ and define $$B = \mathcal{G} (lk(v))$$ where $lk(v)$ is the link of $v$ in $\Gamma$ i.e. the subgraph of $\Gamma$ spanned by the vertices adjacent to $v$. Let $\Gamma_1$ be the subgraph of $\Gamma$ spanned by the vertices $V(\Gamma) \setminus \{ v \}$, $\Gamma_2 = star(v)$ be the star of $v$ i.e. the subgraph of $\Gamma$ spaned by $v$ and $lk(v)$. We set $$G_1 = \mathcal{G}(\Gamma_1) \hbox{ and }G_2 = \mathcal{G} (\Gamma_2).$$ Note that $$G_2 =  B \times G(v).$$ The decomposition $\Gamma = \Gamma_1 \cup \Gamma_2$, where $\Gamma_1 \cap \Gamma_2 = lk(v)$ gives the decomposition as a free amalgamated pro-$\C$ product
\begin{equation} \label{amalgam1} G = G_1 \amalg_B G_2. \end{equation}
Note that the above free amalgamated pro-$\C$ product is proper, since both $G_1$ and $G_2$ are pro-$\C$ retracts of $G$ (one easily deduces it from \cite[Proposition 9.2.2]{R-Z2}).
Define $T_v$  to be  the standard pro-$\C$ tree that corresponds to this decomposition.

Note that we have a decomposition
\begin{equation} \label{eq-q} \mathcal{G}(\Gamma) = \mathcal{G}(\Gamma^{(1)}) \times \ldots \times \mathcal{G}(\Gamma^{(k)})
\end{equation}
that cannot be further decomposed.
For a normal subgroup $N$ of $\mathcal{G}(\Gamma)$ we say that $N$ is full, if $N$ intersects non-trivially each factor $\mathcal{G}(\Gamma^{(i)})$ in the above decomposition.

We write $Z(V(\Gamma))$ for all $ v \in V(\Gamma)$ that are linked by edge with any vertex in  $V(\Gamma)$.

 \begin{lemma} \label{core1} Let $\Gamma$ be a finite (simplicial) graph and $\Gamma_0$ be a subgraph of $\Gamma$. Then the core of $\mathcal{G}(\Gamma_0)$ in $G = \mathcal{G}(\Gamma)$ is $\mathcal{G}(\Delta)$, where $\Delta$ is the max subgraph of $\Gamma_0$ such that $G$ has a decomposition as $\mathcal{G} (\Delta) \times \mathcal{G} (\Delta_1)$, where $\Delta$ and $\Delta_1$ could be the empty sets.
\end{lemma}

\begin{proof} We induct on $| V(\Gamma) |$. If $\Gamma_0 = \Gamma$ there is nothing to prove. Thus we can assume that $| V(\Gamma) \setminus V(\Gamma_0)| \geq 1$.

Suppose first that $V(\Gamma) \setminus V(\Gamma_0) = \{ v_0 \}$.
Let $\Gamma_1 = lk(v_0)$. Then we have a decomposition as a free amalgamated pro-$\C$ product
$$G = \mathcal{G}(\Gamma_0) \coprod_{\mathcal{G} (\Gamma_1)} \mathcal{G} (\Gamma_1) \times G(v_0).$$
By Lemma \ref{core} the core of $\mathcal{G}(\Gamma_0)$ in $G$ is the core of $\mathcal{G} (\Gamma_1)$ in $G$.
Note that the core of $\mathcal{G}(\Gamma_1)$ in $G$ is the core of  $\mathcal{G}(\Gamma_1)$ in  $\mathcal{G}(\Gamma_0)$.
 By induction the core of $\mathcal{G}(\Gamma_1)$ in $\mathcal{G}(\Gamma_0)$ is $\mathcal{G}(\Delta_2)$, where $\Delta_2$ is the maximal subgraph of $\Gamma_1$ such that $\mathcal{G}(\Gamma_0) = \mathcal{G}(\Delta_2) \times \mathcal{G}(\Delta_3)$.
Note that $G=\mathcal{G}(\Gamma) = \mathcal{G}(\Delta_2) \times \mathcal{G}(\Delta_4)$, where $\Delta_4$ is the subgraph of $\Gamma$ spanned by $\Delta_3$ and $v_0$.
 Thus $\Delta_2 = \Delta$,  $\Delta_1 = \Delta_4$.

Suppose that  $| V(\Gamma) \setminus V(\Gamma_0) | \geq 2$. Let $\Gamma_2$ be a subgraph of $\Gamma$ that contains $\Gamma_0$ and $V(\Gamma) \setminus V(\Gamma_2) = \{ v_0 \}$.
The core of $\mathcal{G}(\Gamma_0)$ in $\mathcal{G} (\Gamma_2)$ is $\mathcal{G} (\Delta_5)$, where $\Delta_5$ is the maximal subgraph of $\Gamma_0$ such that we have a decomposition $\mathcal{G} (\Gamma_2) = \mathcal{G} (\Delta_5) \times \mathcal{G} (\Delta_6)$.

 Let $\Gamma_3$ be the subgraph of $\Gamma$ spanned by $\Delta_5$ and $v_0$. Then the core of $\mathcal{G}(\Gamma_0)$ in $G$ is the core of $\mathcal{G} (\Delta_5)$ in $G$. And it is    the core  of $\mathcal{G} (\Delta_5)$ in $\mathcal{G} (\Gamma_3)$  that by induction    is $\mathcal{G}(\Delta_7)$, where $\Delta_7$ is the maximal subgraph of $\Delta_5$ such that $\mathcal{G} (\Gamma_3) = \mathcal{G}(\Delta_7) \times \mathcal{G}(\Delta_8)$. Then $G = \mathcal{G}(\Delta_7) \times \mathcal{G}(\Delta_9)$, where $\Delta_9$ is the subgraph of $\Gamma$ spanned by $\Delta_6$ and $\Delta_8$. Then $\Delta = \Delta_7$ and $\Delta_1 = \Delta_9$.
\end{proof}

\begin{lemma} \label{faith-irr}  Let $\Gamma$ be a finite (simplicial) graph and $G = \mathcal{G} (\Gamma) $ be indecomposible as a direct product i.e. in \eqref{eq-q} $k = 1$.  Then $G$ acts faithfully and irreducibly on  the standard pro-$\C$ tree $T_v$ associated to (\ref{amalgam1}).
\end{lemma}

\begin{proof} The faithfulness is equivalent to $\cap_{g \in G} B^g = 1$. Then we apply Lemma \ref{core1}.

Note that the indecomposability of $G$ implies that $Z(V) \not= \emptyset $.
The irreducibility comes from the fact that  $T_v/ G$ is  just one edge with two vertices.
\end{proof}

\begin{theorem} \label{graph-product} Let $\Gamma$ be a finite simplicial graph. Suppose that  for every   central vertex $w \in V(\Gamma)$ every proper quotient of $G(w)$ is a finite-by-abelian  pro-$\C$ group. Suppose further that $N$ is a non-trivial, finitely generated, normal, full pro-$\C$ subgroup of $\mathcal{G}(\Gamma)$. Then $\mathcal{G}(\Gamma)/ N$ is  finitely generated, finite-by-abelian, in particular it is abelian-by-finite.
\end{theorem}

\begin{proof}
Consider the decomposition (\ref{eq-q}). Suppose  $\Gamma^{(i)}$ has more than one vertex and let $v_{i} \in V(\Gamma^{(i)})$ be an arbitrary vertex. Note that  $Z(\Gamma^{(i)}) = \emptyset$ otherwise the decomposition (\ref{eq-q}) can be further decomposed. Consider $N_i$ the image of $N$ in $\mathcal{G}(\Gamma^{(i)}) = : G_i$. By assumption of the fullness of $N$ we have that $N_i$ is non-trivial.

 Consider the decomposition as pro-$\C$ amalgamated product
$$G_i = \mathcal{G}(\Gamma_{0,i}) \coprod_{\mathcal{G} (\Gamma_{1,i})} (\mathcal{G} (\Gamma_{1,i}) \times G(v_{i}) ),$$
where  $\Gamma_{1,i}$ is the link of $v_{i}$ in $\Gamma^{(i)}$ and $\Gamma_{0,i} = \Gamma^{(i)} \setminus \{ v_{i} \}$.

Let  $T_{v_i}$ be the standard  pro-$\C$ tree associated to the above decomposition. Then $G_i$ acts on  $T_{v_i}$ and  by  Lemma \ref{faith-irr}  $G_i$ acts faithfully and irreducibly on  $T_{v_i}$. Then by Theorem \ref{main}  $T_{v_i}/ N_i$ is finite  or $G_i$ is virtually procyclic.
 Note that $G$ acts on $T_{v_i}$  via the canonical projection $G \to G_i$. Thus if  $T_{v_i}/ N_i$ is finite we have that $T_{v_i}/ N \prod_{j \not=i} G_j$ is finite.

Suppose that $G_i$ is virtually procyclic. By assumption $v_{i} \in V(\Gamma^{(i)}) \setminus Z(\Gamma^{(i)})$, hence there is $w_i \in  V(\Gamma^{(i)})$ such that $v_{i}$ and $w_i$ are not linked by an edge, so $G_i$ maps epimorphically to $G(v_{i}) \coprod G(w_i)$ and $G(v_{i}) \coprod G(w_i)$ is virtually procyclic. Then $G(v_{i}) = C_2 = G(w_i)$ and $\Gamma^{(i)}$ is just one edge with vertices $v_{i}$ and $w_i$. Thus $G_i =  C_2 \coprod C_2$  and every non-trivial normal closed subgroup of $G_i$ is open, in particular $[G_i : N_i] < \infty$. Then $[G : N \prod_{j \not=i} G_j]$ is finite, so  $T_{v_i}/ N \prod_{j \not=i} G_j$ is finite.

 Note that  $ (\prod_{j \not= i} G_j)  \mathcal{G} ( lk_{\Gamma^{(i)}}(v_{i})) = \mathcal{G} (lk_{\Gamma}(v_i))$. Then  $T_{v_i}/ N \prod_{j \not=i} G_j$  is finite is equivalent to \begin{equation} \label{index1}  [G : N  \mathcal{G} (lk_{\Gamma}(v_i))] < \infty,\end{equation}
  where $G = \mathcal{G}(\Gamma)$.  The above holds for every  $v_i \in V(\Gamma^{(i)})$ with $| V(\Gamma^{(i)})| > 1$. Furthermore  $| V(\Gamma^{(i)})| = 1$ if and only if  $ V(\Gamma^{(i)}) \subseteq Z(V(\Gamma))$.

Then we can continue as in \cite{CR-Z}, where the abstract case is considered. For completeness we outline the proof.

Consider the short exact sequence
$$1 \to \frac{N \prod_{v \in Z(\Gamma)} [G(v), G(v)]}{N} \to \frac{G}{N} \to \frac{G}{N \prod_{v \in Z(\Gamma)}  [G(v), G(v)]} \to 1.$$
Since $N$ is full,  $N \cap G(v) \not=1$ for every $v \in Z(\Gamma)$. Hence $G(v)/ (N \cap G(v))$ is finite-by-abelian and the commutator subgroup of $\frac{N \prod_{v \in Z(\Gamma)} G(v)}{N}$ is finite, so $\frac{N \prod_{v \in Z(\Gamma)} [G(v), G(v)]}{N}$ is finite.

Define $$H = \cap_{v \in V(\Gamma) \setminus Z(\Gamma)} N \mathcal{G} ( lk(v))$$
and note that by (\ref{index1}) $[G : H] < \infty$.
We claim that $M: = \frac{G}{N \prod_{v \in Z(\Gamma)}  [G(v), G(v)]}$ is finite-by-abelian, hence $\frac{G}{N}$ is finite-by-abelian. To prove that $M$ is finite-by-abelian it suffices to show that $K : =
 \frac{H \prod_{v \in Z(\Gamma)}  [G(v), G(v)]}{N \prod_{v \in Z(\Gamma)}  [G(v), G(v)]}$ is central in $M$ together with the observation that $[M : K] < \infty$.

  Suppose that $v \in V(\Gamma) \setminus Z(\Gamma)$,
 $h \in H$, $g_v \in G(v)$, $h = n h_v$, where $n \in N, h_v \in \mathcal{G}(lk(v))$.
 Note that
$$[h, g_v] = [n h_v, g_v] = [n, g_v]^{h_v}. [h_v, g_v] = [n, g_v]^{h_v} \in N^{h_v} = N,$$
thus
$$[H, G(v)] \subseteq N \hbox{ for } v  \in V(\Gamma) \setminus Z(\Gamma).$$
Finally $$[H, G(v)] \subseteq [G, G(v)] = [G(v), G(v)] \hbox{ for
 } v \in Z(\Gamma)$$ imply $[K,M] = 1$, as claimed. Thus $M$  is finitely generated, central-by-finite, hence by a pro-$\C$ version of a result of Schur \cite{Schur} $M$ is a finite-by-abelian pro-$\C$ group.

Finally note that every finitely generated, finite-by-abelian group is abelian-by-finite ( in both discrete and pro-$\C$ case).
\end{proof}

\section{Normal subgroups of pro-$p$ RAAGs with abelian quotients}

From now on we shall deal with pro-$p$ groups only.

If $\Gamma_0$ is a subgraph of a graph $\Gamma$, we say that $\Gamma_0$ is dominant in $\Gamma$ if for every vertex $v \in V(\Gamma) \setminus V(\Gamma_0)$ there is a vertex $w = w(v) \in V(\Gamma_0)$ such that $v$ and $w$ are connected by an edge.

\begin{theorem} \label{dominant} Let $\Gamma$ be a finite simplicial graph and $G = G_{\Gamma}$ be the corresponding pro-$p$ RAAG. Let $\chi : G \to \mathbb{Z}_p$ be an epimomorphism of pro-$p$ groups. Then $N = Ker (\chi)$ is finitely generated as a pro-$p$ group if and only if
the subgraph $\Gamma(\chi) $ of $\Gamma$ spanned by $\{ v \in V(\Gamma) \ | \ \chi(v) \not= 0 \}$ is connected and dominant in $\Gamma$.
\end{theorem}

\begin{proof}
1)  Suppose that $N$ is finitely generated as a pro-$p$ group.

1.1) Suppose that $\Gamma(\chi)$ is not connected. Then we have a decomposition $\Gamma(\chi) = \Gamma_1 \cup \Gamma_2$, where the subgraphs $\Gamma_1$ and $\Gamma_2$ are not connected by an edge.
Consider the epimorphism of groups $$\pi : G = G_{\Gamma} \to G_{\Gamma(\chi)}$$ that  is the identity on  $G(v)$ for $v \in V(\Gamma(\chi))$ and sends  $G(w)$ to $1$ for $\chi(w) = 0$.
Note that $Ker (\pi) \subseteq N$ and let $N_0 = \pi(N)$. Then
$N_0$ is a finitely generated normal pro-$p$ subgroup  of
$$G_0 : = G_{\Gamma(\chi)} = G_{\Gamma_1} \amalg G_{\Gamma_2}$$
 Note that $G_0$ is not virtually procyclic.
Then either  $N_0 = 1$, hence  $G_0 \simeq G_0/ N_0 \simeq  \mathbb{Z}_p$ , a contradiction  or
 by  Corollary  \ref{free product}  we get   $\infty = | \mathbb{Z}_p| = [G_0 : N_0] < \infty$,  a contradiction.  This proves that $\Gamma(\chi)$ is connected.

1.2) Suppose that $\Gamma(\chi)$ is not dominant in $\Gamma$. Let $v $ be a vertex of $\Gamma \setminus \Gamma(\chi)$ that is not connected by an edge with a vertex from $\Gamma(\chi)$. Then $G_{star(v)} \subseteq N$ and $|V(\Gamma)| \geq 2$, hence $G$ is not virtually procyclic.
Consider the decompostion
$$G = G_{\Gamma \setminus \{ v \} } \amalg_{G_{lk(v)}} G_{star(v)},$$
 set $A  =  G_{\Gamma \setminus \{ v \} }, B = G_{lk(v)}, C =  G_{star(v)}$.  Let $T$ be the standard pro-$p$ tree associated to this decomposition. Since $T/ G$ is just one edge with two vertices, the action of $G$ on $T$ is irreducible. The kernel of the action is the core of $B$ in $G$ i.e. $\cap_{g \in G} B^g$.
By Theorem \ref{main} either $[G : NB] < \infty$ or $N$ is in the kernel of the action, in particular $N \subseteq B$. The latter is false since $v \in N \setminus B$  and if the former holds since $B \leq G_{star(v)} \leq N$, we have  $NB = N$, hence $\infty > [G : NB] = [G : N] = | \mathbb{Z}_p| = \infty$, a contradiction.  This completes the proof of the fact that $\Gamma(\chi)$ is  dominant in $\Gamma$.

2) Suppose that $\Gamma(\chi) $ is connected and dominant in $\Gamma$. We aim to prove that $N = Ker(\chi)$ is finitely generated ( as a pro-$p$ group). We induct on the number of vertices in $\Gamma$.

Choose a maximal subtree of $\Gamma(\chi)$ and extend it to a maximal subtree  $\Gamma_0$ of $\Gamma$  in such a way that $\Gamma_0(\chi)$ is dominant in $\Gamma_0$.  Note that there is an epimorphism of  pro-$p$ groups $\pi : G_{\Gamma_0} \to G_{\Gamma}$ that is identity on vertices. Let $\chi_0 = \chi \circ \pi$. If $N_0 = Ker(\chi_0)$ is finitely generated as a  pro-$p$ group then $N = Ker(\chi) = \pi(N_0)$ is finitely generated as required.
Thus we can assume from now on that $\Gamma = \Gamma_0$ is a finite tree.

Choose a pending vertex  $v \in V(\Gamma)$.  Then consider the decomposition as a free amalgamated pro-$p$ product
$$
G = G_{\Gamma} = G_{\Gamma_1} \coprod_{G_{lk(v)}} G_{star(v)},
$$
where $\Gamma_1$ is the subgraph of $\Gamma$ spanned by $V(\Gamma) \setminus \{ v \}$. Then by the Mayer-Vietoris  exact sequence (see \cite[Section 9.4]{R}) we have
$$ \ldots \to \oplus_{g \in G/N G_{star(v)} } H_1(N \cap G_{star(v)}^g, \mathbb{F}_p) \oplus \oplus_{g \in G/ N G_{\Gamma_1}} H_1(N \cap G_{\Gamma_1}^g, \mathbb{F}_p) \to H_1(N, \mathbb{F}_p)$$ $$ \to  \oplus_{g \in G/N G_{lk(v)} } H_0(N \cap G_{lk(v)}^g, \mathbb{F}_p) \to \ldots
$$

Note that since $v$ is a pending vertex we have that $lk(v)$ is just a vertex $w$ and the condition that $\Gamma(\chi)$ is connected and dominant in $\Gamma$ implies  $\chi(w) \not= 0$.
 Then $G_{lk(v)} = \overline{\langle w \rangle} $ is not a subset of $N$ and since $G/ N \simeq \mathbb{Z}_p$ we conclude that $[G :  N G_{lk(v)}]$ is finite, hence $[G :  N G_{star(v)}]$ and $[G : N G_{\Gamma_1}]$ are finite. By induction $N \cap G_{star(v)}$ and $N \cap G_{\Gamma_1}$ are finitely generated hence $N \cap G_{star(v)}^g = (N \cap G_{star(v)})^g$ and $N \cap G_{\Gamma_1}^g = (N \cap G_{\Gamma_1})^g$ are finitely generated. Then in the above exact sequence $H_1(N, \mathbb{F}_p)$ is sandwiched between finite modules, hence it is finite too. This implies that $N$ is finitely generated as a pro-$p$ group.
\end{proof}
 {\bf Remark}  The proof of one of the directions of Theorem \ref{dominant} works for pro-$\C$ RAAGs, namely if $N = Ker (\chi)$ is finitely generated, where $\chi : G \to \mathbb{Z}_{\C}$ is an epimorphism and $G= G_{\Gamma}$ is the pro-$\C$ RAAG associated to $\Gamma$ then $\Gamma(\chi)$ is connected and dominant in $\Gamma$.

In the proof of the next theorem we shall use the following easy result.

\begin{lemma}\label{finite generatedness} Let
$$A\to G\to G/A$$ be an exact sequence of pro-$p$ groups with $A$ abelian and $G/A$ abelian finitely generated. Then $G$ is finitely generated if and only if $A$ is a finitely generated  $\Z_p[[G/A]]$-module,  where the $G/ A$-action on $A$ is induced by conjugation.\end{lemma}

\begin{proof} Let $g_1,\ldots g_n\in G$  be elements such that $g_1A,\ldots, g_nA$ generate $G/A$. If $a_1,\ldots, a_m$ are generators of $A$   as a $\mathbb{Z}_p[[G/A]]$-module,   then clearly $a_1,\ldots a_m, g_1\ldots g_n$ generate $G$.

Conversly,
as a part of 5-term Hoschild-Serre exact sequence we have an exact sequence
$H_2(G/A,\F_p)\rightarrow (A/pA)_{G/A}\rightarrow  H_1(G , \mathbb{F}_p)$. As $G/A$ is finitely presented the left term is finite and so is the right term since $G$ is finitely generated. Hence  the middle term is finite and it remains to observe that $A$ is finitely generated as  $\Z_p[[G/A]]$-module if and only if the module of coinvariants $(A/pA)_{G/A}=A/pA  \widehat{\otimes}_{\F_p[[G]]} \F_p$ is finite.

\end{proof}

\begin{theorem} \label{fin-gen}
Let $\Gamma$ be a finite simplicial graph and $G = G_{\Gamma}$ be the corresponding pro-$p$ RAAG. Let $N$ be a pro-$p$  normal subgroup of $G$ such that $G/ N$ is  infinite and abelian. Then $N$ is finitely generated if, and only if,  $N_0$ is finitely generated for every $  N_0 \triangleleft G$ containing $N$ with $G/ N_0 \simeq \mathbb{Z}_p$ .
\end{theorem}

\begin{proof}
I) Suppose that $N$ is finitely generated. Since $N_0 / N$ is a pro-$p$ subgroup of the finitely generated abelian pro-$p$ group $G/ N$, the quotient group  $N_0/N$ is finitely generated. Hence $N_0$ is finitely generated.

II) Suppose that  every $N_0 \subseteq G$ containing $N$  with $G/ N_0 \simeq \mathbb{Z}_p$ is finitely generated.
We induct on the number of vertices in $\Gamma$ to show that  $N$ is finitely generated.

Fix a vertex $v \in V(\Gamma)$ and consider the decomposition
$$
G = G_{\Gamma} = G_{\Gamma_1} \coprod_{G_{lk(v)}} G_{star(v)}
$$
where  $\Gamma_1$ is the subgraph of $\Gamma$ spanned by $V(\Gamma) \setminus \{ v \}$.  Put $G_1=G_{\Gamma_1}$.

Consider $N_1 = N \cap G_1$.  Suppose first that $N_1 \not= G_1$.
Let  $\chi_0 : G_1 \to \mathbb{Z}_p$ be an epimorphism of  pro-$p$ groups
such that $\chi_0(N \cap G_1) = 0$.  We extend $\chi_0$ to
an epimorphism $\chi : G \to \mathbb{Z}_p$ such that $\chi(v) = 0$.
By assumption $Ker(\chi)$ is finitely generated, so by Theorem \ref{dominant} $\Gamma(\chi)$ is connected and dominant in $\Gamma$. By construction $\Gamma_1(\chi_0) = \Gamma(\chi)$, so $\Gamma_1(\chi_0)$ is connected and dominant in $\Gamma_1$. Then by Theorem \ref{dominant} $Ker(\chi_0)$ is finitely generated. By induction $N_1 = N \cap G_1$  is finitely generated.

If $N_1 = G_1$ then $N_1$ is obviously finitely generated.

We write $G'$ for the pro-$p$ subgroup of $G$ generated by $\{ [x,y] = x^{-1} y^{-1} xy \ | \ x,y \in G \}$ and $G''$ for the pro-$p$ subgroup defined as $(G')'$.
Note that $G'\subseteq N$, hence $G'' \subseteq N'$ and $N$ is finitely generated if and only if $N/ N'$ is finitely generated. This implies that $N$ is finitely generated if and only if $N/ G''$ is finitely generated. Set $$A = G'/ G'', A_1 = G_1'/ G_1'', Q = G/ G' \hbox{  and } Q_1 = G_1 / G_1'.$$
We identify $Q_1$ with a pro-$p$ subgroup of $Q$. We view $A$ as a right pro-$p$ $\mathbb{Z}_p[[Q]]$-module, where the $Q$-action is given by conjugation and in the same way, $A_1$ is a right pro-$p$  $\mathbb{Z}_p[[Q_1]]$-module, where the $Q_1$-action is given by conjugation.  By Lemma \ref{finite generatedness} $N/G''$ is finitely generated  if and only if  $A$ is finitely generated as a pro-$p$ $\mathbb{Z}_p[[\pi(N)]]$-module, where $\pi : G \to {Q} = G/ G' $ is the canonical projection.
Since $N_1$ is finitely generated, by Lemma \ref{finite generatedness} $A_1$ is finitely generated as pro-$p$ $\mathbb{Z}_p[[\pi(N_1)]]$-module.

Let $V$ be the pro-$p$ $\mathbb{Z}_p[[Q]]$-submodule of $A$ generated by the image of $A_1$. Actually $A_1$ embeds in $A$ but we are not going to use this. Then $V$ is a quotient of the pro-$p$ $ \mathbb{Z}_p[[Q]]$-module  $W = A_1 \widehat{\otimes}_{\mathbb{Z}_p[[Q_1]]} \mathbb{Z}_p[[Q]]$,
 where $\widehat{\otimes}_{\mathbb{Z}_p[[Q_1]]}$ denotes the completed tensor product in the category of ${\mathbb{Z}_p[[Q_1]]}$-modules. 

We assume first that $\pi(N) $ is not inside $Q_1$ i.e. $[Q : Q_1 \pi(N)] < \infty$ and hence $[ Q/ \pi(N) : Q_1/ \pi(N_1)] < \infty$.
Then we have that
$$W \widehat{\otimes}_{\mathbb{Z_p[[\pi(N)]]}} \mathbb{Z}_p =  A_1 \widehat{\otimes}_{\mathbb{Z}_p[[Q_1]]} \mathbb{Z}_p[[Q]] \widehat{\otimes}_{\mathbb{Z_p[[\pi(N)]]}} \mathbb{Z}_p
\simeq A_1 \widehat{\otimes}_{\mathbb{Z}_p[[Q_1]]} \mathbb{Z}_p[[Q/ \pi(N)]]
$$
$$
\simeq  (A_1 \widehat{\otimes}_{\mathbb{Z}_p[[\pi(N_1)]]} \mathbb{Z}_p) \widehat{\otimes}_{\mathbb{Z}_p[[Q_1/ \pi(N_1)]]} \mathbb{Z}_p[[Q/ \pi(N)]] =: M
$$
Since $A_1$ is finitely generated as a pro-$p$  $\mathbb{Z}_p[[\pi(N_1)]]$-module, we have that
$A_1  \widehat{\otimes}   _{\mathbb{Z}_p[[\pi(N_1)]]} \mathbb{Z}_p$ is finitely generated as a pro-$p$ $\mathbb{Z}_p$-module. This combined with $[ Q/ \pi(N) : Q_1/ \pi(N_1)] < \infty$ implies that $M$ is finitely generated as a pro-$p$ $\mathbb{Z}_p$-module. Then $W$ is finitely generated as a pro-$p$ $\mathbb{Z}_p[[\pi(N)]]$-module, hence \begin{equation} \label{eq2} V \hbox{ is finitely generated as a pro-}p \    \mathbb{Z}_p[[\pi(N)]]\hbox{-module.} \end{equation}

Our aim is to show that $A$ is finitely generated as a pro-$p$ $\mathbb{Z}_p[[\pi(N)]]$-module. By (\ref{eq2}) this is equivalent with $A_0 = A/ V$ is finitely generated as a pro-$p$ $\mathbb{Z}_p[[\pi(N)]]$-module.

Consider $$G_0 = G_1^{ab} \coprod_{B^{ab}} (B^{ab} \times \overline{\langle v \rangle}).$$ Then
$G_0'/ G_0'' \simeq A_0$ and thus the problem is reduced to showing the original claim but for the pro-$p$ RAAG  $G_0$. Thus we can assume from now on that $G_1$ is abelian, hence $G_1 = C \times B$ and $$G = (C \times B) \coprod_B (B \times \overline{\langle v \rangle}) = B \times (C \coprod  \overline{\langle v \rangle}) ,$$ where $\mathbb{Z}_p \simeq \overline{\langle v \rangle}$.
 Let $$\rho : G \to G/ B \simeq C \coprod  \overline{\langle v \rangle}$$ be the canonical projection. Then there is a short exact sequence of pro-$p$ groups $$ 1 \to B \cap N \to N \to \rho(N) \to 1.$$ Note that $B \cap N$ is a pro-$p$ subgroup of $B \simeq \mathbb{Z}_p^m$ for some $m$, hence $B \cap N$ is finitely generated. We claim that $\rho(N)$ has finite index in $C \coprod  \overline{\langle v \rangle}$, hence is finitely generated. Indeed since $G/ N$ is abelian, $C \coprod  \overline{\langle v \rangle}/ \rho(N)$ is abelian. If $[C \coprod  \overline{\langle v \rangle} : \rho(N)] = \infty$ then there is a pro-$p$ subgroup $N_0$ of $G$ such that $ N \subseteq N_0$, $G/ N_0 \simeq \mathbb{Z}_p$
and $C \coprod  \overline{\langle v \rangle}/ \rho(N_0) \simeq \mathbb{Z}_p$. Hence $\rho(N_0)$ is not finitely generated, a contradiction with the fact that $N_0$ is finitely generated by assumption.

It remains to consider the case when $\pi(N)$ is inside any possible $Q_1$ i.e. $N = G'$. By Theorem \ref{dominant} $\Gamma(\chi)$ is connected and dominant in $\Gamma$ for every epimorphism $\chi: G \to \mathbb{Z}_p$. We claim that this implies that $\Gamma$ is the complete graph on $V(\Gamma)$ and hence $G$ is a free abelian pro-$p$ group. Indeed, if $\Gamma$ is not the complete graph, there is a decomposition $\Gamma = \Gamma_1 \cup star(v)$ with $star(v) \cap \Gamma_1 = lk(v)$ and $\Gamma_1 \not= lk(v)$. Then consider  an epimorphism $\chi: G \to \mathbb{Z}_p$  with $\chi(\Gamma_1 \setminus lk(v)) \not= 0$ and $\chi(star(v)) = 0$. Then $\Gamma(\chi)$ is not dominant in $\Gamma$, a contradiction.
\end{proof}

\section{Classifying coabelian pro-$p$ subgroups of pro-$p$ RAAGs depending on  their homological type $FP_n$ :  a pro-$p$ approach }

 Let $G = G_{\Gamma}$ be the pro-$p$ RAAG associated to a finite simplicial graph $\Gamma$. Consider the complex of  finitely generated free $\mathbb{F}_p[[G]]$-modules,
$$\mathcal{F}_{\Gamma} : \ldots \to P_n \to P_{n-1} \to \ldots \to P_0 \to \mathbb{F}_p \to 0,
$$
where $$P_n = \oplus_{ \Sigma \subset V(\Gamma)  \hbox{ clique} , | \Sigma| = n}  c_{\Sigma} \mathbb{F}_p[[G]],$$
the direct sum is over all cliques $\Sigma$, including the empty set and $c_{\Sigma} \mathbb{F}_p[[G]]$ is free cyclic $\mathbb{F}_p[[G]]$-module.  Recall that a clique is a subset $\Sigma$ of $V(\Gamma)$ such that every two elements of $\Sigma$ are linked by an edge in $\Gamma$.  For a clique $\{ v_1, \ldots, v_n \}$ we use the notation  $(v_1, \ldots, v_n)$ to mean that $v_1 < \ldots < v_n$.
For $\Sigma = (v_1, \ldots, v_n)$ the diferential is given by
$$\partial_n (c_{\Sigma}) = \sum_{1 \leq i \leq n} (-1)^{i-1} c_{\Sigma \setminus \{ v_i \} } (v_i - 1)$$
and $\partial_0(c_{\emptyset})=1$.

\begin{lemma} \label{exact} $\mathcal{F}_{\Gamma}$ is an exact complex.
\end{lemma}

\begin{proof}
We induct on the cardinality of the set of vertices $V(\Gamma)$ of $\Gamma$.

We fix a vertex $v \in V(\Gamma)$ and define $B =  G_{lk(v)}  $. Let $\Gamma_1$ be the subgraph of $\Gamma$ spanned by the vertices $V(\Gamma) \setminus \{ v \}$, $\Gamma_2 = star(v)$. We set $G_1 = G_{\Gamma_1}$ and $G_2 = G_{\Gamma_2}.$  The decomposition $\Gamma = \Gamma_1 \cup \Gamma_2$, where $ \Gamma_0 = \Gamma_1 \cap \Gamma_2 = lk(v)$ gives the decomposition as a free amalgamated pro-$p$ product
$$ G = G_1 \amalg_B G_2. $$
This is a particular case of (\ref{amalgam1}) where we think of $G_{\Gamma}$ as a graph product of pro-$p$ groups with each $G(w) = \overline{\langle w \rangle} \simeq \mathbb{Z}_p$ for $w \in V(\Gamma)$. By induction the complexes    $\mathcal{F}_{\Gamma_0}$,$\mathcal{F}_{\Gamma_1}$ and $\mathcal{F}_{\Gamma_2}$  are exact.

We claim that there is a short exact sequence of complexes
\begin{equation}\label{eqqq}  0 \to \mathcal{F}_{\Gamma_0} \widehat{\otimes}_{\mathbb{F}_p[[B]]} \mathbb{F}_p[[G]]  \stackrel{\alpha}\to
 \mathcal{F}_{ \Gamma_1}\widehat{\otimes}_{\mathbb{F}_p[[G_1]]} \mathbb{F}_p[[G]] \oplus  \mathcal{F}_{ \Gamma_2} \widehat{\otimes}_{\mathbb{F}_p[[G_2]]} \mathbb{F}_p[[G]]  \stackrel{\beta}\to  \mathcal{F}_{\Gamma} \to 0\end{equation}
that in dimension  -1 is the pro-$p$ tree $T_v$ i.e. is the exact sequence of $\mathbb{F}_p[[G]]$-modules
$$0 \to \mathbb{F}_p \widehat{\otimes}_{\mathbb{F}_p[[B]]} \mathbb{F}_p[[G]] \to  \mathbb{F}_p \widehat{\otimes}_{\mathbb{F}_p[[G_1]]} \mathbb{F}_p[[G]] \oplus  \mathbb{F}_p \widehat{\otimes}_{\mathbb{F}_p[[G_2]]} \mathbb{F}_p[[G]] \to \mathbb{F}_p \to 0,$$
where $\alpha(c_{\Sigma} \widehat{\otimes} \lambda) = ( c_{\Sigma} \widehat{\otimes} \lambda, - c_{\Sigma} \widehat{\otimes} \lambda)$ and $\beta(c_{\Sigma_1} \widehat{\otimes} \lambda_1, c_{\Sigma_2} \widehat{\otimes} \lambda_2) = c_{\Sigma_1} \lambda_1 + c_{\Sigma_2} \lambda_2$.
Indeed (\ref{eqqq}) is equivalent to the short exact sequence
$$
0 \to \oplus_{\Sigma \subseteq V(\Gamma_0) \hbox{ clique}, | \Sigma| = n } c_{\Sigma} \F_p[[G]] \stackrel{\alpha_n}\to
 (\oplus_{\Sigma \subseteq V(\Gamma_1) \hbox{ clique}, | \Sigma| = n } c_{\Sigma} \F_p[[G]] ) \oplus (  \oplus_{\Sigma \subseteq V(\Gamma_2) \hbox{ clique}, | \Sigma| = n } c_{\Sigma} \F_p[[G]] )
$$ \begin{equation} \label{eq-q1} \stackrel{\beta_n}\to  \oplus_{\Sigma \subseteq V(\Gamma) \hbox{ clique}, | \Sigma| = n } c_{\Sigma} \F_p[[G]] \to 0,\end{equation}
where $\alpha_n(c_{\Sigma}\lambda) = ( c_{\Sigma} \lambda, - c_{\Sigma}\lambda)$ and
$\beta_n( c_{\Sigma_1} \lambda_1, c_{\Sigma_2} \lambda_2) = c_{\Sigma_1} \lambda_1 + c_{\Sigma_2} \lambda_2$ for $\lambda, \lambda_1, \lambda_2 \in \F_p[[G]]$. Note that (\ref{eq-q1}) follows from the fact that if $\Sigma$ is a clique in $V(\Gamma)$ then either $\Sigma$ is a clique in $V(\Gamma_1)$ or $\Sigma$ is a clique in $V(\Gamma_2)$ and that $\Gamma_0 = \Gamma_1 \cap \Gamma_2$.

Since  $\mathcal{F}_{ \Gamma_0}$,$\mathcal{F}_{ \Gamma_1}$ and $\mathcal{F}_{ \Gamma_2}$ are exact complexes, we deduce that $ \mathcal{F}_{ \Gamma_0} \widehat{\otimes}_{\mathbb{F}_p[[B]]} \mathbb{F}_p[[G]]$, $ \mathcal{F}_{ \Gamma_1} \widehat{\otimes}_{\mathbb{F}_p[[G_1]]} \mathbb{F}_p[[G]]$ and $ \mathcal{F}_{ \Gamma_2} \widehat{\otimes}_{\mathbb{F}_p[[G_2]]} \mathbb{F}_p[[G]]$ are exact complexes.
Consider the long exact sequence in homology associated to a short exact sequences of complexes
$$ \ldots \to 0 = H_n(  \mathcal{F}_{ \Gamma_1} \widehat{\otimes}_{\mathbb{F}_p[[G_1]]} \mathbb{F}_p[[G]] \oplus  \mathcal{F}_{ \Gamma_2} \widehat{\otimes}_{\mathbb{F}_p[[G_2]]} \mathbb{F}_p[[G]]) \to $$ $$ H_n(\mathcal{F}_{ \Gamma}) \to  H_{n-1} ( \mathcal{F}_{ \Gamma_0}\widehat{ \otimes}_{\mathbb{F}_p[[B]]} \mathbb{F}_p[[G]]) = 0 \to \ldots $$
Then $H_n(\mathcal{F}_{ \Gamma}) = 0$ for any $n$ and we conclude that $\mathcal{F}_{ \Gamma}$ is an exact complex, as required.
\end{proof}

Recall that  the flag complex $\Delta(\Gamma)$ associated to the graph $\Gamma$ is the complex obtained from $\Gamma$ by gluing
a simplex $(v_1, \ldots , v_n)$ for every non-empty clique $\{ v_1, \ldots, v_n \}$ of $\Gamma$, where $ v_1 < v_2 < \ldots < v_n$.
 
 Let $\Delta (\Gamma(\chi))$ be the flag complex of the graph $\Gamma(\chi)$ and $\Delta ( \Gamma)$  be the flag complex of the graph $\Gamma$. Let $S$ be a clique em $\Gamma$. The link $lk_{\Delta(\Gamma)}(S)$ is the complex of simplices corresponding to cliques $v$  such that $v \cup S $ is a clique em $\Gamma$, here we permit $S$ to be the empty set. For a clique $S \subset \Gamma \setminus \Gamma(\chi)$ we set
 $lk_{\Delta(\Gamma(\chi))} (S) = lk_{\Delta(\Gamma)} (S) \cap \Delta(\Gamma(\chi))$.

We give a proof of the following theorem that does not require using discrete groups, later we will prove a generalization in Theorem \ref{thm-link} where the condition $V(\Gamma) \subseteq Ker (\chi) \cup \chi^{-1} (\mathbb{Z}_p \setminus p \mathbb{Z}_p)$ will be dropped at the expence of the need   to use  related non-trivial results for discrete groups.

\begin{theorem}  \label{mild-cond} Let $\Gamma$ be a finite simplicial graph and $G = G_{\Gamma}$ be the corresponding pro-$p$ RAAG. Let  $\chi : G \to \mathbb{Z}_p$ be an epimorphism of pro-$p$ groups such that $V(\Gamma) \subseteq Ker (\chi) \cup \chi^{-1} (\mathbb{Z}_p \setminus p \mathbb{Z}_p)$.
Then $N=Ker(\chi)$ is of type $FP_n$ if and only if $lk_{\Delta(\Gamma(\chi))} ( S )$ is $(n-1- |S|)$-acyclic over $\mathbb{F}_p$ for every clique $S$ from $\Gamma \setminus \Gamma(\chi)$  such that $n-1- |S| \geq 0$. For $S = \emptyset$ this
translates to the flag complex $\Delta ( \Gamma(\chi))$ is $(n-1)$-acyclic over $\mathbb{F}_p$.
\end{theorem}
\begin{proof}
Let $\chi_0 : G \to \mathbb{Z}_p$ be the homomorphism of pro-$p$ groups given by $\chi_0(v) = 1$ if $\chi(v) \not= 0$ and $\chi_0(v) = 0$ if $\chi(v) = 0$ for  $v \in V(\Gamma)$.
 Then there is an automorphism $\varphi : G \to G$ such that $\chi = \chi_0 \circ \varphi$ and $\varphi( Ker (\chi)) = Ker(\chi_0)$. This is possible since we can set $\varphi(v) = v^{\chi(v)}$ if $\chi(v) \not= 0$ and $\varphi(v) = v$ if $\chi(v) = 0$.
 Note that we have used here  the condition  $V(\Gamma) \subseteq Ker (\chi) \cup \chi^{-1} (\mathbb{Z}_p \setminus p \mathbb{Z}_p)$, since for $\varphi$ to be an automorphism we need that $\chi(v) \in \mathbb{Z}_p^* = \mathbb{Z}_p \setminus p \mathbb{Z}_p$ for $\chi(v) \not= 0$.

 Then $Ker (\chi)$ and $Ker (\chi_0)$ have the same homological type  $FP_n$. Note  that $\Gamma (\chi) = \Gamma(\chi_0)$, hence $lk_{\Delta(\Gamma(\chi))} ( S )  = lk_{\Delta(\Gamma(\chi_0))} (S) $ for a clique $S \in \Gamma \setminus \Gamma(\chi)$. Thus we can assume from now on that $\chi = \chi_0$.

 Recall that for a clique $\{ v_1, \ldots, v_n \}$ we use the notation  $(v_1, \ldots, v_n)$ to mean that $v_1 < \ldots < v_n$.
Consider the complex $\widetilde{\mathcal{F}}_{\Gamma}  = \mathcal{F}_{\Gamma} \widehat{\otimes}_{\mathbb{F}_p[[N]]} \mathbb{F}_p$
$$\widetilde{\mathcal{F}}_{\Gamma} :  \ldots \to Q_n = \oplus_{\Sigma \subseteq V(\Gamma) \hbox{ clique}, | \Sigma| = n} c_{\Sigma} \mathbb{F}_p[[G/ N]] \to \ldots
$$
with the differential given for $\Sigma = ( v_1, \ldots, v_n )$ by
$$\widetilde{\partial}_n (c_{\Sigma}) =  \sum_{\chi(v_i) \not= 0} (-1)^{i-1} c_{\Sigma \setminus \{ v_i \} } (v_i N-1) =
\sum_{\chi(v_i) \not= 0} (-1)^{i-1} c_{\Sigma \setminus \{ v_i \} } (g-1),
$$
where
$g = v_iN = v_j N$ for any $v_i, v_j \in V(\Gamma)  \setminus Ker (\chi)$;  here $g$ is  a fixed generator of $G/ N \simeq \mathbb{Z}_p$.

Consider the complex
$$\mathcal{C} : \ldots \to C_n = \oplus_{\Sigma \subseteq V(\Gamma), \hbox{ clique}, | \Sigma| = n} c_{\Sigma} \mathbb{F}_p \to \ldots \to C_0 \to \mathbb{F}_p \to 0
$$
with differential
$$d_n (c_{\Sigma}) = \sum_{\chi(v_i) \not= 0} (-1)^{i-1} c_{\Sigma \setminus \{ v_i \} }, d_n(c_{\emptyset})=1$$
where $\Sigma = ( v_1, \ldots, v_n )$. We identify $\widetilde{\mathcal{F}}_{\Gamma}$ with $\mathcal{C} \widehat{\otimes}_{\mathbb{F}_p} \mathbb{F}_p[[G/N]]$. Then $\widetilde{\partial}_n (c_{\Sigma}) = d_n(c_{\Sigma})\widehat{\otimes} (g-1)$ for $n\geq 1$.

Define $A_n = Ker ({d}_n) \widehat{\otimes}_{\mathbb{F}_p} \mathbb{F}_p[[G/N]]$, $B_n = Im ({d}_{n+1}) \widehat{\otimes}_{\mathbb{F}_p} \mathbb{F}_p[[G/ N]] $ and $D_n = Im ({d}_{n+1}) \widehat{\otimes}_{\mathbb{F}_p} (g-1) \mathbb{F}_p[[G/ N]]  $.
Then we have a short exact sequence
$$0 \to B_n / D_n \to A_n / D_n \to A_n / B_n \to 0$$
Note that $Im ({d}_{n+1})$ is a finite $\mathbb{F}_p$-module, hence $$B_n / D_n \simeq Im (d_{n+1}) \widehat{\otimes}_{\mathbb{F}_p} ( \mathbb{F}_p[[G/N]]/ (g-1) \mathbb{F}_p[[G/N]]) \simeq  Im (d_{n+1}) \widehat{\otimes}_{\mathbb{F}_p} \mathbb{F}_p  \simeq  Im (d_{n+1})$$  is finite.
Furthermore $$H_n(N, \mathbb{F}_p) \simeq H_n (\widetilde{\mathcal{F}}_{\Gamma}) \simeq A_n/ D_n.$$ Since $N = Ker (\chi)$ is of type $FP_n$ if and only if $H_i(N, \mathbb{F}_p)$ is finite for all $i \leq n$, we conclude that $N$ is of type $FP_n$ if and only if $A_i/ B_i$ is finite for all $i \leq n$.

By the definition of $A_n$ and $B_n$ we have that
$$A_n/ B_n \simeq H_n(\mathcal{C}) \widehat{\otimes}_{\mathbb{F}_p} \mathbb{F}_p[[G/N]],$$
hence $A_i/ B_i$ is finite for all $i \leq n$ if and only if $ H_i(\mathcal{C}) = 0$ for all $ 1 \leq i \leq n$.

For a clique $S = \{ v_{k+1}, \ldots, v_n \} \in \Gamma \setminus \Gamma(\chi)$ (i.e. $\chi(v_i) = 0$ for all $ k+1 \leq i \leq n$) we write $\mathcal{C}^S$ for the subcomplex of $\mathcal{C}$ spanned by all $c_{\Sigma}$ such that $\Sigma = \{ v_1, \ldots, v_k, v_{k+1} , \ldots, v_n \}$ is a clique with $\chi(v_i) \not= 0$ for $ 1 \leq i \leq k$. Then
$$H_n(\mathcal{C}) \simeq \oplus_{S \subseteq \Gamma \setminus \Gamma(\chi) \hbox{ clique}  , |S| \leq n-1} H_n(\mathcal{C}^S).$$

Note that $\mathcal{C}^S$ is obtained from the chain complex of $lk_{\Delta(\Gamma(\chi))} (S)$ with coeficients in $\mathbb{F}_p$ by shifting the dimension with $|S|+1$. Then
$$H_n(\mathcal{C}^S) \simeq H_{n-1- |S|}(lk_{\Delta(\Gamma(\chi))} (S), \mathbb{F}_p). $$
\end{proof}

 {\bf Example.} Suppose $\Gamma$ is a circuit of length $n \geq 4$, $G_{\Gamma}$ the pro-$p$ RAAG associated to $\Gamma$ and $\chi : G_{\Gamma} \to \mathbb{Z}_p$ with $\chi(v) \not= 0 $ for every $v \in V(\Gamma)$. Then $\Delta(\Gamma(\chi)) = \Gamma(\chi) = \Gamma$ is $0$-acyclic but not $1$-acyclic over $\mathbb{F}_p$. Thus $Ker(\chi)$ is finitely generated but not finitely presented ( as a pro-$p$ group).

\section{Classifying coabelian pro-$p$ subgroups of pro-$p$ RAAGs depending on  their homological type $FP_n$ : an approach via abstract groups }

\begin{lemma} \label{exact2} Let $Q^{dis} = \mathbb{Z}^k$, $Q = \mathbb{Z}_p^k$ and $V$ is a finite $\mathbb{F}_p[ Q^{dis}]$-module. Then $M = V \otimes_{\mathbb{F}_p[ Q^{dis}]} \mathbb{F}_p[[Q]]$ is finite.
\end{lemma}

\begin{proof}
Since $V$ has a filtration of $\mathbb{F}_p[ Q^{dis}]$-submodules with quotients that are cyclic $\mathbb{F}_p[ Q^{dis}]$-modules, we can assume that $V$ is a cyclic $\mathbb{F}_p[ Q^{dis}]$-module i.e. $V = \mathbb{F}_p[ Q^{dis}]/ I$ for some ideal $I$ of $\mathbb{F}_p[ Q^{dis}]$.

Note that $\mathbb{F}_p[[Q]] = \mathbb{F}_p[[t_1, \ldots, t_k]]$ and $\mathbb{F}_p[ Q^{dis}] = \mathbb{F}_p[t_1, \ldots, t_k, 1/ ( t_1 + 1), \ldots, 1/ (t_k + 1)]$.
Since $V$ is finite there is $0 \not= f_i \in I \cap \mathbb{F}_p[t_i]$, $t_i + 1 $ does not divide $f_i$ for $ 1 \leq i \leq k$. Then $M$ is a quotient of $M_1 = \mathbb{F}_p[[t_1, \ldots, t_k]]/ (\sum_{1 \leq i \leq k} f_i \mathbb{F}_p[[t_1, \ldots, t_k]])$.

Finally since
$\mathbb{F}_p[[t_1, \ldots, t_k]]$ is the inverse limit of $\mathbb{F}_p[t_1, \ldots, t_k]/ (t_1, \ldots, t_k)^m$ we conclude that $M_1$ is the inverse limit of  $\mathbb{F}_p[t_1, \ldots, t_k]/ ((t_1, \ldots, t_k)^m, f_1, \ldots , f_k)$.
But $\dim_{\mathbb{F}_p} \mathbb{F}_p[t_1, \ldots, t_k]/ ((t_1, \ldots, t_k)^m, f_1, \ldots , f_k)$ $ \leq$ $ s : = \dim_{\mathbb{F}_p} \mathbb{F}_p[t_1, \ldots, t_k]/ (f_1, \ldots , f_k)
$ and $s$ does not depend on $m$. Hence $dim_{\mathbb{F}_p} M_1 \leq s$.
\end{proof}

\begin{theorem}  \label{thm-link} Let $\Gamma$ be a finite simplicial graph and $G = G_{\Gamma}$ be the corresponding pro-$p$ RAAG. Let  $\chi : G \to \mathbb{Z}_p$ be an epimorphism of pro-$p$ groups and  $N = Ker(\chi)$.
Then $N$ is of type $FP_n$ if and only if $lk_{\Delta(\Gamma(\chi))} (S)$ is $(n-1- |S|)$-acyclic over $\mathbb{F}_p$ for every clique $S$ from $\Gamma \setminus \Gamma(\chi)$ with $ n-1- |S| \geq 0$. For $S = \emptyset$ this
translates to the flag complex $\Delta ( \Gamma(\chi))$ is $(n-1)$-acyclic over $\mathbb{F}_p$.
\end{theorem}

\begin{proof}
 Consider the automorphism $\varphi : G \to G$ given by $\varphi(v_i) = v_i^{k_i}$ if $\chi(v_i) \not= 0$, $\chi(v_i) = k_i p^j$, $k_i \in \mathbb{Z}_p \setminus p \mathbb{Z}_p$, $j \geq 0$ and $\varphi(v_i) = v_i$ if $\chi(v_i) = 0$.
 Thus  $\chi = \chi_0 \circ \varphi$ , where $\chi_0(V(\Gamma)) \subseteq \{ 0,1 , p, p^2, \ldots, p^m, \ldots \}$ and  $\varphi( Ker (\chi)) = Ker(\chi_0)$.
 So $Ker (\chi)$ and $Ker (\chi_0)$ have the same homological type  $FP_n$. Note that $\Gamma (\chi) = \Gamma(\chi_0)$, hence $lk_{\Delta(\Gamma(\chi))} ( S)  = lk_{\Delta(\Gamma(\chi_0))} (S) $ for a clique $S \subset \Gamma \setminus \Gamma(\chi)$. Thus we can assume from now on that $\chi = \chi_0$, thus $\chi(V(\Gamma)) = \chi_0(V(\Gamma)) \subseteq \mathbb{Z}$.

Let $G^{dis}$ be the discrete RAAG associated to the graph $\Gamma$. Let $\chi^{dis} : G^{dis} \to \mathbb{Z}$ be the epimorphism whose restriction to $V(\Gamma)$ coincides with $\chi$.

1) Assume that $lk_{\Delta(\Gamma(\chi))} (S)$ is $(n-1- |S|)$-acyclic over $\mathbb{F}_p$ for every clique $S$ from $\Gamma \setminus \Gamma(\chi)$.

By \cite[Thm. 9]{B-G}  $lk_{\Delta(\Gamma(\chi))} (S)$ is $(n-1- |S|)$-acyclic over $\mathbb{Z}$ for every clique $S$ from $\Gamma \setminus \Gamma(\chi)$ if and only if $N^{dis} = Ker (\chi^{dis})$ is of type $FP_n$ over $\mathbb{Z}$. The same works with $\mathbb{Z}$ substituted with $\mathbb{F}_p$. Hence $N^{dis}$ is of type $FP_n$ over $\mathbb{F}_p$ and so $H_i(N^{dis}, \mathbb{F}_p)$ is finite for all $ i \leq n$.

Consider the complex of  finitely generated $\mathbb{F}_p[G^{dis}]$-modules
$$\mathcal{F}_{\Gamma}^{dis} : \ldots \to P_n^{dis} \to P_{n-1}^{dis} \to \ldots \to P_0^{dis} \to \mathbb{F}_p \to 0
$$
where $$P_n^{dis} = \oplus_{ \Sigma \subset V(\Gamma) \hbox{ clique } , |\Sigma| = n} c_{\Sigma} \mathbb{F}_p[G^{dis}],$$
the direct sum is over all cliques $\Sigma$, including the empty set and $c_{\Sigma} \mathbb{F}_p[G^{dis}]$ is free cyclic $\mathbb{F}_p[G^{dis}]$-module.
For $\Sigma = ( v_1, \ldots, v_n )$ the diferential is given by
$$\partial_n (c_{\Sigma}) = \sum_{1 \leq i \leq n} (-1)^{i-1} c_{\Sigma \setminus \{ v_i \} } (v_i - 1).$$
Note that $\mathcal{F}_{\Gamma}^{dis}$ is a free resolution of $\mathbb{F}_p$ as $\mathbb{F}_p[G^{dis}]$-module.

Consider the complex $\widetilde{\mathcal{F}}_{\Gamma}^{dis}  = \mathcal{F}_{\Gamma}^{dis} \otimes_{\mathbb{F}_p[N^{dis}]} \mathbb{F}_p$. Then
$$H_i(N^{dis}, \mathbb{F}_p) \simeq H_i(\widetilde{\mathcal{F}}_{\Gamma}^{dis} ) \hbox{ is finite for } i \leq n.$$
Set $Q = G/N \simeq \mathbb{Z}_p$ and $Q^{dis} = G^{dis}/ N^{dis} \simeq \mathbb{Z}$.
Note that
$$\widetilde{\mathcal{F}}_{\Gamma} \simeq \widetilde{\mathcal{F}}_{\Gamma}^{dis} \otimes_{\mathbb{F}_p[ Q^{dis}]} \mathbb{F}_p[[Q]]$$

Since $- \otimes_{\mathbb{F}_p[ Q^{dis}]} \mathbb{F}_p[[Q]]$ is an exact functor, by Lemma \ref{exact2}
$$H_i(N, \mathbb{F}_p) \simeq H_i (\widetilde{\mathcal{F}}_{\Gamma}) \simeq H_i(\widetilde{\mathcal{F}}_{\Gamma}^{dis} \otimes_{\mathbb{F}_p[ Q^{dis}]} \mathbb{F}_p[[Q]]) \simeq  H_i(\widetilde{\mathcal{F}}_{\Gamma}^{dis}) \otimes_{\mathbb{F}_p[ Q^{dis}]} \mathbb{F}_p[[Q]] \hbox{ is finite for } i\leq n,$$
 because $H_i(\widetilde{\mathcal{F}}_{\Gamma}^{dis} ) $ is finite for $i \leq n$. Thus $N$ is a pro-$p$ group of type $FP_n$.

 2) Suppose now that $N$ is a pro-$p$ group of type $FP_n$. Then as above
 $$H_i(N^{dis}, \mathbb{F}_p) \otimes_{\mathbb{F}_p[ Q^{dis}]} \mathbb{F}_p[[Q]] \simeq H_i(\widetilde{\mathcal{F}}_{\Gamma}^{dis}) \otimes_{\mathbb{F}_p[ Q^{dis}]} \mathbb{F}_p[[Q]] \simeq H_i(N, \mathbb{F}_p)   \hbox{ is finite for } i\leq n.$$

 Note that since $G^{dis} $ is of type $FP_n$ we have that $H_i(N^{dis}, \mathbb{F}_p)$ is finitely generated $R$-module, where $R : = {\mathbb{F}_p[ Q^{dis}]} $ is a principal ideal domain. Hence $H_i(N^{dis}, \mathbb{F}_p)$ is a direct sum of cyclic $R$-modules (see \cite[Theorem 16.1.6]{curtis-rainer}). We claim that $H_i(N^{dis}, \mathbb{F}_p)$ is finite. Otherwise $R$ is a direct summand of $H_i(N^{dis}, \mathbb{F}_p)$, hence $\mathbb{F}_p[[Q]]$ is a direct summand of the finite abelian group $H_i(\widetilde{\mathcal{F}}_{\Gamma}^{dis}) \otimes_{\mathbb{F}_p[ Q^{dis}]} \mathbb{F}_p[[Q]]$, a contradiction.

 Since  $H_i(N^{dis}, \mathbb{F}_p)$ is finite for $i \leq n$ by \cite[Cor. 5.2]{P-S} $lk_{\Delta(\Gamma(\chi))} (S)$ is $(n-1- |S|)$-acyclic over $\mathbb{F}_p$ for every clique $S$ from $\Gamma \setminus \Gamma(\chi)$, as required.
\end{proof}

  Let $N$ be a normal pro-$p$ subgroup of the right angled pro-$p$ group $G = G_{\Gamma}$ such that $G/ N$ is abelian. Recall that $G$ is the pro-$p$ completion of $G^{dis}$ and that we call $N$ discretely embedded in $G$ if there is a subgroup $N^{dis}$ of $G^{dis}$ such that $G^{dis}/ N^{dis}$ is abelian and $N$ is the closure of $N^{dis}$ in $G$. We call $N$ weakly discretely embedded in $G$ if there is an automorphism $\varphi$ of $G$ such that $\varphi(N)$ is discretely embedded in $G$.

We note that the group of automorphisms $Aut(G)$ is much bigger than the group of automorphisms $Aut(G^{dis})$.  For example every homomorphism of pro-$p$ groups $\varphi : G \to G$ given by $\varphi(v) = v^{z_v}$ for $v \in V(\Gamma)$  for some $z_v \in \mathbb{Z}_p \setminus p \mathbb{Z}_p$, is an automorphism. A set of generators of $Aut(G^{dis})$ is described in \cite{L}.

\begin{theorem} \label{thm-final} Let $\Gamma$ be a finite simplicial graph and $G = G_{\Gamma}$ be the corresponding pro-$p$ RAAG. Let $N$ be a normal pro-$p$ subgroup of $G$ such that $G/ N$ is infinite abelian, and assume further that  $N$ is weakly discretely embedded in $G$. Then $N$ is of type $FP_n$ if and only if $N_0$ is  of  type $FP_n$ for  every $N_0 \subseteq G$ containing $N$ with $G/ N_0 \simeq \mathbb{Z}_p$.
\end{theorem}

\begin{proof}
Suppose that for every $N \subseteq N_0 \subseteq G$ with $G/ N_0 \simeq \mathbb{Z}_p$ we have that $N_0$ is $FP_n$. We aim to show that $N$ is a pro-$p$ group of type $FP_n$.
If $\varphi \in Aut(G)$ we have that $\varphi(N) \simeq N$ have the same homological type $FP_n$, thus we can assume that $N$ is discretely embedded in $G$.

 1) Let $N^{dis} \leq N_0^{dis} \leq G^{dis}$ be discrete groups such that $G^{dis}/N_0^{dis} \simeq \mathbb{Z}$, $G^{dis}/ N^{dis}$ is abelian and $N$ is the closure of $N^{dis}$ in $G$. Let $N_0$ be the closure of $N_0^{dis}$ in $G$, thus $G/ N_0 \simeq \mathbb{Z}_p$.  Assume any such $N_0$ is a pro-$p$ group of type $FP_n$. By Theorem \ref{thm-link} $lk_{\Delta(\Gamma(\chi))} (S)$ is $(n-1- |S|)$-acyclic over $\mathbb{F}_p$ for every clique $S$ from $\Gamma \setminus \Gamma(\chi)$, where $\chi : G \to G/ N_0 = \mathbb{Z}_p$ is the canonical epimorphism.

Let $\chi_0: G^{dis} \to G^{dis}/ N_0^{dis} = \mathbb{Z}$ be the canonical epimorphism. Note that $\Gamma(\chi) = \Gamma(\chi_0)$. By \cite[Thm. 9]{B-G} $lk_{\Delta(\Gamma(\chi_0))} (S)$ is $(n-1- |S|)$-acyclic over $\mathbb{Z}$ for every clique $S$ from $\Gamma \setminus \Gamma(\chi_0)$ if and only if $N_0^{dis}$ is of type $FP_n$ ( over $\mathbb{Z}$). The same proof shows that $lk_{\Delta(\Gamma(\chi_0))} (S)$ is $(n-1- |S|)$-acyclic over $\mathbb{F}_p$ for every clique $S$ from $\Gamma \setminus \Gamma(\chi_0)$ if and only if $N_0^{dis}$ is of type $FP_n$ over $\mathbb{F}_p$. Thus  $N_0^{dis}$ is of type $FP_n$  over $\mathbb{F}_p$.

It was shown in \cite[Thm 3.5]{K-MP} that if each $N_0^{dis}$ with $G^{dis}/ N_0^{dis} \simeq \mathbb{Z}$  is of type $FP_n$, then $N^{dis}$ is of type $FP_n$. The version over $\mathbb{F}_p$ of this property is that if each $N_0^{dis}$ with $G^{dis}/ N_0^{dis} \simeq \mathbb{Z}$   is of type $FP_n$ over $\mathbb{F}_p$, then $N^{dis}$ is of type $FP_n$ over $\mathbb{F}_p$ and the proof is the same as in \cite{K-MP} with the obvious change of basic ring $\mathbb{Z}$ to $\mathbb{F}_p$. Thus $H_i(N^{dis}, \mathbb{F}_p)$ is finite for all $i \leq n$.

Let $Q = G/ N$. By substituting $N$ with a commensurable pro-$p$ subgroup of $G$ that contains the commutator pro-$p$ subgroup we can assume that $Q \simeq \mathbb{Z}_p^k$.
Note that $- \otimes_{\mathbb{F}_p[ Q^{dis}]} \mathbb{F}_p[[Q]]$ is an exact functor, so using the notation from the proof of Theorem \ref{thm-link} we have
$$H_i(N, \mathbb{F}_p) \simeq H_i (\widetilde{\mathcal{F}}_{\Gamma}) \simeq H_i(\widetilde{\mathcal{F}}_{\Gamma}^{dis} \otimes_{\mathbb{F}_p[ Q^{dis}]} \mathbb{F}_p[[Q]]) \simeq  H_i(\widetilde{\mathcal{F}}_{\Gamma}^{dis}) \otimes_{\mathbb{F}_p[ Q^{dis}]} \mathbb{F}_p[[Q]].$$
 Recall that $H_i(\widetilde{\mathcal{F}}_{\Gamma}^{dis} ) \simeq  H_i(N^{dis}, \mathbb{F}_p) $ is finite for $i \leq n$. Then by Lemma \ref{exact2} $H_i(N, \mathbb{F}_p)$ is finite for $i \leq n$, hence  $N$ is a pro-$p$ group of type $FP_n$.

 2) The converse, (i.e. if $N$ is of type $FP_n$ then each $N_0$ as described in Theorem \ref{thm-final} is $FP_n$), is obvious since $N_0/ N$ is a finitely generated abelian pro-$p$ group, hence is of type $FP_{\infty}$, in particular is $FP_n$.

\end{proof}

\end{document}